\documentclass[10pt,reqno]{amsart}

\usepackage{amsmath}
\usepackage{amsthm}
\usepackage{mathtools}
\usepackage[T1]{fontenc}
\usepackage{mathrsfs}
\usepackage{latexsym}
\usepackage{amssymb}
\usepackage{cancel}
\usepackage[normalem]{ulem}
\usepackage{bm} 
\usepackage{mathabx}
\usepackage{graphicx}
\usepackage{float}
\usepackage{extarrows}
\usepackage{epstopdf}
\usepackage{caption}
\usepackage{subcaption}
\usepackage{enumerate}   
\usepackage{tikz}
\usepackage{xcolor}
\usepackage[a4paper,
left=2cm, right=2cm,
top=3cm, bottom=3cm]{geometry}
\floatstyle{boxed}
\restylefloat{figure}
\usepackage{hyperref}
\hypersetup{
    colorlinks,
    linkcolor={lblue},
    citecolor={lblue},
    urlcolor={violet}
}

\definecolor{lblue}{rgb}{0.0, 0.313, 0.608}
\definecolor{lred}{rgb}{1.0,0.5,0.5}

\newcommand{\N}{\mathbb{N}}		   
\newcommand{\Z}{\mathbb{Z}}				
\newcommand{\C}{\mathbb{C}}	
\newcommand{\T}{\mathbb{T}}	
\newcommand{\R}{\mathbb{R}}	

\newcommand{\e}{\varepsilon}
\newcommand{\I}{\mathrm{i}}

\DeclarePairedDelimiter\abs{\lvert}{\rvert}

\newtheorem{definition}{Definition}[section]
\newtheorem{thm}[definition]{Theorem}
\newtheorem{cor}[definition]{Corollary}
\newtheorem{lem}[definition]{Lemma}
\newtheorem{prop}[definition]{Proposition}
\newtheorem{rem}[definition]{Remark}
\numberwithin{equation}{section}

\DeclareMathOperator{\clos}{clos}
\DeclareMathOperator{\dist}{dist}

\DeclareMathOperator{\Id}{Id}
\DeclareMathOperator{\range}{range}
\DeclareMathOperator{\spa}{span}

\begin{document}

\title{Traveling waves for a quasilinear wave equation}

\author{Gabriele Bruell}
\address{Centre for Mathematical Sciences, Lund University, 22100 Lund, Sweden}
\email{gabriele.brull@math.lth.se}

\author{Piotr Idzik}
\address{PII Pipetronix GmbH, Lorenzstra{\ss}e 10, 76297 Stutensee, Germany}
\email{piotr.idzik@bakerhughes.com}

\author{Wolfgang Reichel}
\address{Institute for Analysis, Karlsruher Institute of Technology (KIT), D-76128 Karlsruhe, Germany}
\email{wolfgang.reichel@kit.edu}

\date{\today }

\subjclass[2000]{Primary: 35B32, 35L72; Secondary: 35C07, 35Q61}
\keywords{nonlinear Maxwell equations, quasilinear wave equation, traveling wave, bifurcation}

\begin{abstract}  We consider a 2+1 dimensional  wave equation appearing in the context of polarized waves for the nonlinear Maxwell equations. The equation is quasilinear in the time derivatives and involves two material functions $V$ and $\Gamma$. We prove the existence of traveling waves which are periodic in the direction of propagation and localized in the direction orthogonal to the propagation direction. Depending on the nature of the nonlinearity coeffcient $\Gamma$ we distinguish between two cases: (a) $\Gamma\in L^\infty$ being regular and (b) $\Gamma=\gamma\delta_0$ being a multiple of the delta potential at zero. For both cases we use bifuraction theory to prove the existence of nontrivial small-amplitude solutions. One can regard our results as a persistence result which shows that  guided modes known for linear wave-guide geometries survive in the presence of a nonlinear constitutive law. Our main theorems are derived under a set of conditions on the linear wave operator. They are subsidised by explicit examples for the coefficients $V$ in front of the (linear) second time derivative for which our results hold. 
\end{abstract}

\maketitle


\section{Introduction}

Of concern is the following 2+1 dimensional quasilinear wave equation


\begin{equation}\label{eq:Q}
-\Delta u+\partial_t^2\left(V(\lambda,y)u+\Gamma(y)u^3\right)=0,
\end{equation}
 which appears in the context of polarized waves for the nonlinear Maxwell equations. Here, $u=u(t,x,y)$ is the unkown depeding on time $t\in \R_+$ and the two spatial variables $x,y$. 
 We assume $u$ to be periodic in $x$-direction and localized in $y$-direction. In what follows, we denote by $\T$ the one dimensional flat $2\pi$-periodic torus, so that $(x,y)\in \T \times \R$. The potentials $V(\lambda,\cdot)$ and $\Gamma$ depend only on $y$ and incorporate material properties. Here, $\lambda \in \R$ is a parameter. The function $\Gamma$  might be a bounded function  (referred to as \emph{regular $\Gamma$}) or a multiple of a delta potential at $y=0$ (referred to as \emph{distributional $\Gamma$}).
 
 \medspace

 To motivate our interest in \eqref{eq:Q}, let us explain how it appears in the context of electromagnetics. Recall that the
%
%
%
Maxwell equations in the absence of charges and currents are given by
\begin{align*}
	\nabla\cdot\mathbf{D}&=0, &\nabla\times\mathbf{E}\,=&-\partial_t\mathbf{B}, &\mathbf{D}=&\varepsilon_0\mathbf{E}+\mathbf{P}(\mathbf{E}), \\
	\nabla\cdot\mathbf{B}&=0, &\nabla\times\mathbf{H}=&\,\partial_t\mathbf{D}, &\mathbf{B}=&\mu_0\mathbf{H}. 
\end{align*}
The modeling of the underlying material is done by making assumptions on the form of the polarization field $\mathbf{P}$. Here, we assume that $\mathbf{P}$ depends instantaneously on the electric field $\mathbf{E}$ as follows 
$$
\mathbf{P}(\mathbf{E})=\varepsilon_0\chi_1(\mathbf{x})\mathbf{E}+\varepsilon_0\chi_3(\mathbf{x})\abs{\mathbf{E}}^2\mathbf{E}
$$ with $\mathbf{x}=(x,y,z)\in\R^3$, cf. \cite{agrawal}, Section~2.3. For simplicity we take $\chi_1, \chi_3$ as given scalar functions instead of the more general matrix/tensor structure of these quantities. The values $\varepsilon_0, \mu_0$ are constant such that $c^2=(\varepsilon_0\mu_0)^{-1}$ and $c$ is the speed of light in vacuum. 
By direct calculations from Maxwell's equations one obtains the second-order quasilinear wave-type equation for the electric field
\begin{align}
	0=\nabla\times\nabla\times\mathbf{E} +\partial_t^2\left( V(\mathbf{x})\mathbf{E}+\Gamma(\mathbf{x})\abs{\mathbf{E}}^2\mathbf{E}\right),
	\label{curlcurl}
\end{align}
where $V(\mathbf{x})=\mu_0\varepsilon_0\left(1+\chi_1(\mathbf{x})\right)$ and $\Gamma(\mathbf{x})=\mu_0\varepsilon_0\chi_3(\mathbf{x})$. The magnetic induction $\mathbf{B}$ can be retrieved from $\nabla\times\mathbf{E}=-\partial_t\mathbf{B}$ by time-integration  and it will satisfy $\nabla\cdot\mathbf{B}=0$ provided it does so at time $t=0$. By assumption the magnetic field is given by $\mathbf{H}=\frac{1}{\mu_0} \mathbf{B}$ and it satisfies $\nabla\times\mathbf{H}=\partial_t\mathbf{D}$. It remains to check that the displacement field $\mathbf{D}$ satisfies the Gauss law $\nabla\cdot\mathbf{D}=0$ in the absence of external charges. This follows directly from the constitutive equation $\mathbf{D}=\varepsilon_0(1+\chi_1(\mathbf{x}))\mathbf{E}+\varepsilon_0\chi_3(\mathbf{x})\abs{\mathbf{E}}^2\mathbf{E}$ and the assumption of the polarized form of the electric field  
$$
\mathbf{E}(\mathbf{x},t)=(0,0,u(t,x,y))^T.
$$
If we assume additionally that $V(\mathbf{x})=V(y)$ and $\Gamma(\mathbf{x})=\Gamma(y)$ then the quasilinear vectorial wave-type  equation \eqref{curlcurl} turns into the scalar equation \eqref{eq:Q} for $u=u(t,x,y)$.

\medskip

We study the existence of traveling wave solutions of \eqref{eq:Q} propagating in $x$-direction for certain classes of potentials $V$. More precisely, we consider potentials of the form
 \begin{equation*}\label{eq:V}
 V(\lambda,y)= \lambda V_0(y)+V_1(y),\qquad \lambda \in \R,
 \end{equation*}
  where $V_0\in L^\infty(\R)$ and $V_1$ is a distribution (e.g a $\delta$-potential). If 
  \begin{equation}\label{eq:trav}
  u(t,x,y)=\Phi(x-t,y)
  \end{equation} is a traveling wave solution of \eqref{eq:Q}, propagating in $x$-direction with wave speed $v=1$, then $\Phi:\T \times \R \to \R$ satisfies 
\begin{equation}\label{eq:Q_trav}
-\Phi_{yy} - (1-\lambda V_0(y)-V_1(y))\Phi_{xx}+\Gamma(y)(\Phi^3)_{xx}=0.
\end{equation}
The parameter $\lambda\in\R$ will serve as a bifurcation parameter. 
 One might ask why we introduced $\lambda$ as a bifurcation parameter in the form $V(\lambda,y)= \lambda V_0(y)+V_1(y)$ and not in the more intuitive form $V(\lambda,y)= \lambda (V_0(y)+V_1(y))$ -- the latter giving $\sqrt{\lambda}$ the meaning of the propagation speed of the traveling wave. The reason lies in the distributional character of $V_1$. The choice of the underlying function spaces for our analysis allows to formulate a (suitably defined, cf. Lemma~
\ref{lem:1D}) bounded inverse of $\tilde L_\lambda$, which is a slightly modified version of the wave operator $L_\lambda = -\partial_y^2 -(1-\lambda V_0(y)-V_1(y))\partial_x^2$. Extending the multiplication with $\lambda$ also to the distributional part $V_1$ causes a difficulties when differentiating wit respect to $\lambda$, cf. Remark~\ref{rem:trouble}.

\medskip

\emph{
Throughout the paper, a function $u$ is called a traveling wave solution of \eqref{eq:Q} if $u$ takes the form \eqref{eq:trav} and  $\Phi$ is a solution of \eqref{eq:Q_trav}, which is periodic in its first and localized in its second component.
}

\medskip

We investigate the existence of nontrivial solutions $\Phi_\lambda$ of \eqref{eq:Q_trav} corresponding to the parameter $\lambda$  under certain assumptions on the potential $V$; thereby providing the existence a nontrivial traveling wave solution $u_\lambda$ of \eqref{eq:Q}. We apply bifurcation theory for the parameter $\lambda$ to find nontrivial solutions of \eqref{eq:Q_trav}. Our aim is to analyze the existence of nontrivial solutions in a way as general as possible, finding key properties of associate linear operators, which guarantee the existence of solutions via bifurcation theory. Eventually, we provide examples of specific potentials $V$ for which these properties can be verified.
In particular, the following cases are under consideration:

\medskip

\begin{itemize}
	\item[(P1)] \textbf{$V$ is a $\delta-$potential on a positive background}, that is $V$ is a distribution of the form
	\[
	V(\lambda, y)= \lambda +\alpha \delta_0(y).
	\] 
	\item[(P2)] \textbf{$V$ is a $\delta-$potential on a step background}, that is $V$ is a distribution of the form
		\[
		V(\lambda,y)=\lambda \textbf{1}_{|y|\geq b}+ \beta  \textbf{1}_{|y|<b}+ \alpha \delta_0(y).
		\]
\end{itemize}

\medskip

Concerning the nonlinear potential $\Gamma$ we distinguish between \textbf{regular $\Gamma$}, that is $\Gamma\in L^\infty(\R)$, and \textbf{distributional $\Gamma$ }, that is $\Gamma=\gamma\delta_0$. The main results for regular $\Gamma$ are shown in Section~\ref{subsec:regular} and for distributional $\Gamma$ in Section~\ref{subsec:distributional}.

\medskip

Let us also comment on related work.	Problem \eqref{eq:Q} has been considered in \cite{PelSimWeinstein} where spatially localized traveling wave solutions
of the 1+1-dimensional quasi-linear Maxwell model were investigated. The authors assume that $V(y)$ is a periodic arrangement of delta potentials. Using a multiple scale ansatz in fast and slow time, the field profile is expanded into infinitely many modes which are quasiperiodic in time (time-periodic both in the fast and slow time variables). Using local bifurcation methods the authors solve a related system which is homotopically linked to the Maxwell problem written as any infinite coupled system. It is not clear if the local bifurcation connects the related system and the Maxwell problem but numerical results support the existence
of spatially localized traveling waves. In \cite{dohnal_doerfler} and \cite{dohnal_romani_2021} another approximation (including error-estimates) of a version of \eqref{eq:Q} with periodic coefficients by finitely many coupled modes near band edges has been performed both analytically and numerially.  

In the studies of the nonlinear Maxwell-system \eqref{curlcurl}, often monochromatic waves $\mathbf{E}(\mathbf{x})= U(\mathbf{x})e^{\I\omega t}+c.c.$ are considered. Since a typical cubic nonlinearity generates higher harmonics, they either need to be neglected (leading to an error), or the constitutive equation for $\mathbf{D}$ is replaced by a time-averaged nonlinearity $\mathbf{D}=\varepsilon_0(1+\chi_1(\mathbf{x}))\mathbf{E}+\varepsilon_0\chi_3(\mathbf{x})\frac{1}{T}\int_0^T\abs{\mathbf{E}}^2\,dt \mathbf{E}$, cf. \cite{stuart_1990, stuart_1993, stuart_zhou_2010, Mederski_2015, BDPR_2016, dohnal_romani_doubly, Mederski_Reichel} and particularly the two survey papers \cite{Bartsch_Mederski_survey, Mederski_survey_2}.

In contrast to the previously cited works, our approach is genuinely polychromatic and does not rely on time-averaged material laws. In our solution ansatz we allow for harmonics of arbitrary order and we treat \eqref{eq:Q} without any approximation. Recently in \cite{kohler_reichel} a similar approach was taken and spatially localized, time-periodic solution of \eqref{eq:Q} were obtained via variational methods. The result of the present paper and \cite{kohler_reichel} are complementary in the following sense: (a) in \cite{kohler_reichel} only distributional $\Gamma$ is considered whereas in the present paper we also allow for regular $\Gamma\in L^\infty$; (b) in the case of distributional $\Gamma$, \cite{kohler_reichel} only treats $V\in L^\infty$ whereas in the present paper we always have a delta potential contributing to $V$. Variational methods as in \cite{kohler_reichel} have the advantage of producing solutions which may be far away from the trivial solution whereas local bifurcation methods as in the present paper produce solutions in the vicinity of zero. On the other hand, the local bifurcation method leads to more precise information about the actual shape of the bifurcating branch of solutions. 

\medspace

\textbf{Outline of the paper.} We close the introduction with a brief outline of the paper. In Section~\ref{S:main_results}, we collect our main results. In Theorem~\ref{main1} and Theorem~\ref{main2} we provide a set of conditions on the linear wave operator guaranteeing the existence of nontrivial traveling wave solutions of \eqref{eq:Q} for regular and distributional $\Gamma$, respectively. Subsequentially, we present in Corollary~\ref{cor:p1}, Corollary~\ref{cor:p2}, and Corollary~\ref{cor:p1_distrib}, Corollary~\ref{cor:p2_distrib} particular examples in the form of $(P1)$ and $(P2)$ for regular and distributional $\Gamma$. In Section~\ref{S:notation} we fix some notation. The remaining sections~\ref{S:Gamma_b}--\ref{S:distrib} are devoted to the proofs of our main results. To be more presice, in Section~\ref{S:Gamma_b} we prove the existence result in Theorem~\ref{main1} for regular $\Gamma$ followed by the proofs of Corollary~\ref{cor:p1} and Corollary~\ref{cor:p2} in Section~\ref{S:regular}, which provide specific examples. Similarly, we prove in Section~\ref{S:Gamma_d} the existence result in Theorem~\ref{main2} for distributional $\Gamma$ and finalize our studies in Section~\ref{S:distrib} with the proofs of Corollary~\ref{cor:p1_distrib} and Corollary~\ref{cor:p2_distrib} on specific examples in the case of distributional $\Gamma$. In the appendix we collect auxiliary results.

\bigskip

\section{Main results} \label{S:main_results}

 We are looking for solutions of \eqref{eq:Q_trav} of the form
\begin{equation}\label{eq:form}
\Phi(x,y;\lambda)=\sum_{k\in \N}\phi_k(y;\lambda)\sin(kx).
\end{equation}

Our analysis is going to be divided into two parts separating the case when $\Gamma$ is regular in the sense that $\Gamma \in L^\infty(\R)$  and the case when $\Gamma$  is distributional and takes the extreme form of a $\delta$-potential.  This is essentially due to the fact that in the former case we are concerned with a nonlinear equation on the domain $\T\times\R$, while in case of $\Gamma$ being a $\delta$-potential the problem can be viewed as a linear equation on $\T\times \R\setminus\{0\}$ -- which can be solved separately -- equipped with a nonlinear boundary condition at $x=0$ induced by the delta potential.

\medskip

\subsection{Main result for regular $\Gamma$} \label{subsec:regular}
Let us start with the definition of a weak solution of \eqref{eq:Q_trav} in the case when $\Gamma \in L^\infty(\R)$.

\begin{definition}[Weak solution in the case of regular $\Gamma$] \label{solution_regular}
		\emph{
	We say that $\Phi\in H^2(\T;L^2(\R))\cap H^1(\T;H^1(\R))$ is a \emph{weak solution} of \eqref{eq:Q_trav} if and only if
	\begin{equation*}
	\int_{\T}\int_{\R} \Phi_y\Psi_y - \left(1-\lambda V_0(y)\right)\Phi_{xx}\Psi\,dy\,dx - \int_\T \langle V_1(\cdot)\Phi_x(x,\cdot), \Psi_x(x,\cdot) \rangle\,dx  
	+ \int_\T\int_\R (\Phi^3)_{xx}\Psi \,dy\,dx=0
	\end{equation*}
	for any $\Psi \in H^1(\T;H^1(\R))$. Here, $\langle \cdot, \cdot \rangle$ is the dual pairing between $H^{-1}(\R)$ and $H^1(\R)$.
}
\end{definition}

\begin{rem} \emph{ 
		\begin{itemize}
			\item[(a)] We consider $V_1$ as a bounded linear operator from $H^1(\R)$ into $H^{-1}(\R)$. When $V_1=\delta_0$ this means that for $f,g\in H^1(\R)$ we have $\langle V_1 f,g\rangle = f(0)g(0)$, i.e., since $f\in C(\R)$ it multiplies $\delta_0$ and generates $f(0)\delta_0$ as a distribution acting on $g$.\\
\item[(b)]  Clearly $\Phi_{xx}\in L^2(\T\times\R)$. We shall see in Section~\ref{S:Gamma_b} (cf. \eqref{infinity}, \eqref{l4} in Lemma~\ref{lem:prop_F}) that also $\Phi\in L^\infty(\T\times\R)$ and $\Phi_x \in L^4(\T\times\R)$ so that $(\Phi^3)_{xx}=3\Phi^2\Phi_{xx}+6\Phi\Phi_x^2\in L^2(\T\times\R)$.
\end{itemize}
}
\end{rem}


\medskip

If $\Gamma \in L^\infty(\R)$, the ansatz in \eqref{eq:form} allows us to reduce the problem of finding  nontrivial solutions of \eqref{eq:Q_trav} to studying spectral properties of the family of linear wave operators
\[
	L_\lambda^k:=-\frac{d^2}{dy^2}+k^2(1-\lambda V_0(y)-V_1(y))\qquad\mbox{for}\quad k\in \N.
\]
We prove the following theorem:

\begin{thm}[Existence of traveling waves for regular $\Gamma$]
\label{main1}
 Assume that $\Gamma \in L^\infty(\R)$,  the potential $V$ is given by
 \[
 V(\lambda,y)=\lambda V_0(y) + V_1(y)
 \]
  and
\begin{itemize}
	\item[$(L0)$] $V_0\in L^\infty(\R)$ and $V_1: H^1(\R)\to H^{-1}(\R)$ is bounded;
	\item[$(L1)$] for every fixed $k\in \N$ and $\lambda\in \R$ the operator $L_\lambda^k: D(L_\lambda^k)\subset L^2(\R)\to L^2(\R)$ is self-adjoint;
	\item[$(L2)$] there exists a wavenumber $k_*\in \N$, a value $\lambda_*\in\R$, and an open interval $I_{\lambda_*}\subset \R$  containing $\lambda_*$ such that zero is an isolated simple eigenvalue of $L_{\lambda_*}^{k_*}$
	and $0\in \rho(L_{\lambda}^k)$ for any $(k,\lambda)\in \N\times I_{\lambda_*}$ with $(k,\lambda)\neq (k_*,\lambda_*)$;
	\item[$(L3)$] if $L_\lambda^k \phi =f$ for some $f\in L^2(\R)$, then
	\[
	\|\phi\|_{L^2(\R)}\lesssim \frac{1}{k^2}\|f\|_{L^2(\R)}\qquad\mbox{and}\qquad	\|\phi^\prime \|_{L^2(\R)}\lesssim \frac{1}{k}\|f\|_{L^2(\R)}
	\]
	uniformly for $\lambda\in I_{\lambda_*}$ and $k\in \N$ sufficiently large.
\end{itemize}
If in addition $V_0$ satisfies the transversality condition
\begin{equation}\label{eq:t}
\langle V_0 \phi^*,\phi^* \rangle_{L^2(\R)} \neq  0,
\end{equation}
	where $\phi^*$ spans the one-dimensional kernel of $L_{\lambda_*}^{k_*}$,
then there exists $\e_0>0$ and a smooth curve through $(0,\lambda_*)$, 

\[
	\{(\Phi(\e),\lambda(\e))\mid |\e|<\e_0\}\subset (H^2(\T;L^2(\R))\cap H^1(\T;H^1(\R)))\times I_{\lambda_*},
	\]
	of nontrivial solutions of \eqref{eq:Q_trav} with 
	\begin{align*}
	\Phi(0)&=0,  && D_\e \Phi(0)(x,y)=\phi^*(y)\sin(k_*x), \\
	\lambda(0)&=\lambda_*, && \dot \lambda(0)=0, \quad \ddot \lambda(0)= -\frac{3}{2}\frac{\int_{\R}\Gamma(y)(\phi^*)^4(y)\,dy}{\int_\R V_0(y)(\phi^*)^2(y)\,dy}.
	\end{align*}
\end{thm}

\begin{rem}\emph{
	The transversality condition \eqref{eq:t} is trivially satisfied if $V_0\geq 0$, $\not \equiv 0$ or $V_0\leq 0$, $\not \equiv 0$.
}
\end{rem}


In Section~\ref{S:Gamma_b} we prove Theorem~\ref{main1}.
There are two main requirements on $L_\lambda^k$ providing the existence of nontrivial solutions via bifurcation theory: The first is that there exists a value $\lambda_*\in\R$ of the bifurcation parameter such that $L_{\lambda_*}^k$ has a one-dimensional kernel if and only if $k=k_*$ for some wave number $k_*\in \N$ (see $(L2)$); this is a necessary bifurcation condition. Secondly, we demand that for any $k\neq k_*$ the self-adjoint operator $L_\lambda^k$ has a spectral gap $(-ck^2,ck^2)$ around zero, which ensures the  decay properties of $\phi_k(\cdot;\lambda)$ (see $(L3)$). Eventually, after we have established Theorem \ref{main1}, we turn to the specific case, when $\Gamma \in L^\infty(\R)$  is regular, $V$ are potentials of the form as in (P1)  and (P2) and formulate tangible assumptions on the triple $(k_*,\lambda_*,\alpha)$ (see \eqref{eq:bfp1}  and \eqref{eq:bfp2} below), which guarantee that conditions $(L0)-(L3)$ of Theorem \ref{main1} are satisfied; thereby proving the existence of nontrivial traveling wave solutions of \eqref{eq:Q}. In particular, we prove the following corollaries.

\begin{cor}[Case P1, regular $\Gamma$]
\label{cor:p1}
Let $\Gamma \in L^\infty(\R)$ and $V(\lambda, y)=\lambda + \alpha \delta_0(y)$. If $k_*\in \N$ and $\lambda_*<1$ are given and $\alpha>0$ is determined from 
\begin{equation}\label{eq:bfp1}
\alpha =\frac{2\sqrt{1-\lambda_*}}{k_*} ,
\end{equation}
then the assumptions in Theorem~\ref{main1} are satisfied with
\[
\phi^* (y)=\sqrt{k_*\sqrt{1-\lambda_*}}e^{-k_*\sqrt{1-\lambda_*}|y|}
\]
and
\begin{align*}
	\Phi(0)&=0,  && D_\e \Phi(0)(x,y)=\phi^*(y)\sin(k_*x), \\
	\lambda(0)&=\lambda_*, && \dot \lambda(0)=0, \quad \ddot \lambda(0)= -\frac{3}{2}k_*^2(1-\lambda_*)\int_{\R} \Gamma(y)e^{-4k_*\sqrt{1-\lambda_*}|y|}\,dy.
\end{align*}
\end{cor}

\begin{cor}[Case P2, regular $\Gamma$]
\label{cor:p2}
Let $\Gamma \in L^\infty(\R)$ and $V(\lambda, y)=\lambda \textbf{1}_{|y|\geq b}+ \beta  \textbf{1}_{|y|<b}+ \alpha \delta_0(y)$. If $k_*\in \N$, $b>0$ and $\beta,\lambda_*<1$ are given and $\alpha>0$ is determined from 
\begin{equation}\label{eq:bfp2}
\alpha =\frac{2\sqrt{1-\beta}}{k^*}\cdot \frac{\sqrt{1-\beta}\sinh(k_*\sqrt{1-\beta}b)+\sqrt{1-\lambda_*}\cosh(k_*\sqrt{1-\beta}b)}{\sqrt{1-\beta}\cosh(k_*\sqrt{1-\beta}b)+\sqrt{1-\lambda_*}\sinh(k_*\sqrt{1-\beta}b)} ,
\end{equation}
then the assumptions in Theorem~\ref{main1} are satisfied.
\end{cor}

\begin{rem} \emph{
		Details on the construction of $\phi^*$ in Corollary \ref{cor:p2} can be taken from Section~\ref{Ss:P2}.
	}
\end{rem}

\medskip

\subsection{Main result for distributional $\Gamma$} \label{subsec:distributional}
Again, we start with the definition of a weak solution of \eqref{eq:Q_trav}, but now in the case when $\Gamma$ is given by a $\delta$-potential. We assume that the function $V_1=W+\alpha\delta_0$ splits into a regular part $W$ and the distributional part $\alpha\delta_0$ so that $V=\lambda V_0 +W+\alpha \delta_0$, where $V_0,W\in L^\infty(\R)$ and $\Gamma = \gamma \delta_0$.

\begin{definition}[Weak solution in the case of distributional $\Gamma$] \label{solution_distributional} 
		\emph{
We say that $\Phi\in H^2(\T;L^2(\R))\cap H^1(\T;H^1(\R))\cap C(\R;H^2(\T))$ is a \emph{weak solution} of \eqref{eq:Q_trav} if and only if
	\[
	\int_{\T}\int_{\R} \Phi_y\Psi_y - \left(1-\lambda V_0(y)-W(y)\right)\Phi_{xx}\Psi\,dy\,dx + \int_\T \bigl(\alpha\Phi_{xx}(x,0)+\gamma (\Phi^3)_{xx}(x,0)\bigr)\Psi(x,0) \,dx=0
	\]
	for any $\Psi \in H^1(\T;H^1(\R))$. 
}
\end{definition}

\begin{rem}
	\emph{
Clearly $\Phi_{xx}(\cdot,0)\in L^2(\T)$ and $\Phi(\cdot,0), \Phi_x(\cdot,0)\in L^\infty(\T)$ so that $(\Phi^3(\cdot,0))_{xx}=3\Phi(\cdot,0)^2\Phi(\cdot,0)_{xx}+6\Phi(\cdot,0)\Phi(\cdot,0)_x^2\in L^2(\T)$. Moreover, $\Psi\in H^1(\T;H^1(\R))\subset H^1(\T\times\R)$ has an $L^2$-trace at $y=0$. 
}
\end{rem}

Note that \eqref{eq:Q_trav}  can be written as a linear partial differential equation on $\T\times \R\setminus\{0\}$ equipped with a nonlinear boundary condition on $\T$:
  \begin{equation}\label{eq:gleichung1}
  	\begin{cases}
  -\Phi_{yy}-(1-\lambda V_0-W)\Phi_{xx}=0, \quad &(x,y)\in\T \times \R\setminus\{0\}, \\
  \Phi_y(x,0_+)-\Phi_y(x,0 _-)=\partial_x^2\left(\alpha \Phi+ \gamma \Phi^3\right)(x,0), \quad &x\in\T. 
  \end{cases}
  \end{equation}
  
As before let 
\[
L_{k}^\lambda:= -\frac{d^2}{dy^2}+k^2(1-\lambda V_0(y)-W(y)-\alpha \delta_0(y))
\]
be a family of linear wave operators and set
 \[
   L_{0,k}^\lambda := -\frac{d^2}{dy^2}+k^2(1-\lambda V_0(y)-W(y))
 \]
 to be the regular part of $L_k^\lambda$.
  We prove the following theorem:
  
  \begin{thm}[Existence of traveling wave for distributional $\Gamma$]
  \label{main2}
   Assume that $\Gamma=\gamma \delta_0$, the potential $V$ is given by
   \[
   V(\lambda,y)=\lambda V_0(y)+W(y)+\alpha \delta_0(y),
   \]
    and
\begin{itemize}
	\item[$(\tilde L0)$] $V_0,W\in L^\infty(\R)$ are even;
	\item[$(\tilde L1)$] there exists an interval $I\subset \R$ such that for every fixed $k\in \N$ and $\lambda\in I$ the operator $L_{0,\lambda}^k: H^2(\R)\subset L^2(\R)\to L^2(\R)$ satisfies $0\in \rho(L_{0,\lambda}^k)$;
	\item[$(\tilde L2)$] there exists a wavenumber $k_*\in \N$, a value $\lambda_*\in\R$, and an open interval $I_{\lambda_*}\subset I\subset \R$ containing $\lambda_*$ such that zero is an isolated simple eigenvalue of $ L_{\lambda_*}^{k_*}$ 
	and $0\in \rho(L_{\lambda}^k)$ for any $(k,\lambda)\in \N\times I_{\lambda_*}$ with $(k,\lambda)\neq (k_*,\lambda_*)$;
	\item[$(\tilde L3)$] there exist $C>0$ such that $\|\phi_k(\cdot;\lambda)\|_{L^2(0,\infty)} \leq C$  uniformly for $\lambda\in I_{\lambda_*}$, $k\in \N$, and where $\phi_k \in H^2(0,\infty)$ satisfies\footnote{The existence and the properties of the functions $\phi_k$ are detailed in Remark~\ref{rem:existenz_decaying_mode}.}
	\[
	L_{0,k}^\lambda \phi_k(y;\lambda)=0\quad \mbox{on}\quad (0,\infty)\quad \mbox{with}\quad \phi_k(0;\lambda)=1.
	\]
\end{itemize}
  If in addition $V_0$ satisfies the transversality condition
  \begin{equation}\label{eq:t2}
  \langle V_0 \phi^*,\phi^* \rangle_{L^2(\R)} \neq 0,
  \end{equation}
  	where $\phi^*$ spans the one-dimensional kernel of $L_{\lambda_*}^{k_*}$,
  then there exists $\e_0>0$ and a smooth curve through $(0,\lambda_*)$ 
  \[
  	\{(\Phi(\e),\lambda(\e))\mid |\e|<\e_0\}\subset (H^2(\T;L^2(\R))\cap H^1(\T;H^1(\R))\cap \{\Phi(\cdot,0)\in H^2(\T)\})\times I_{\lambda_*}
  	\]
  	of nontrivial solutions of \eqref{eq:Q_trav} with
  	\begin{align*}
	\Phi(0)&=0,  && D_\e \Phi(0)(x,y)=\phi^*(y)\sin(k_*x), \\
	\lambda(0)&=\lambda_*, && \dot \lambda(0)=0, \quad \ddot \lambda(0)= -\frac{3\gamma}{4\int_\R V_0(y)(\phi^*)^2(y)\,dy}.
	\end{align*} 
  \end{thm}

\begin{rem}
	\emph{
	The transversality condition \eqref{eq:t2} is trivially satisfied if $V_0\geq 0$, $\not \equiv 0$ or $V_0\leq 0$, $\not \equiv 0$.
}
\end{rem}

\begin{rem}\emph{
		 We can formulate $(\tilde L2)$ entirely in terms of the ansatz-functions $\phi_k(\cdot;\lambda)$ defined in $(\tilde L3)$. To this end notice that $\sigma_{ess}(L_{0,\lambda}^k) = \sigma_{ess}(L_\lambda^k)$, cf. Lemma \ref{lem:spectrum}. Since $0\not\in \sigma(L_{0,\lambda}^k)$ by assumption $(\tilde L1)$ it is clear that $0\in \sigma(L_\lambda^k)$ is characterized by zero being an eigenvalue of $L_\lambda^k$. This, however, in combination with the evenness of $V_0, W$, means that $\phi_k(|y|;\lambda)$ is up to scalar multiples the unique candidate for the eigenfunction and has to satisfy 
$$
	2\phi_k^\prime(0_+;\lambda)+k^2\alpha=0.
$$
Here $2\phi_k^\prime(0_+;\lambda)$ is the jump of the first derivative of the even function $\phi_k(|y|;\lambda)$ at $y=0$.  
Based on this characterization of zero belonging to the spectrum of $L_k^\lambda$ we can replace $(\tilde L2)$ by
    \begin{itemize}
			\item[$(\bar L2)$]  there exists a wavenumber $k_*\in \N$, $\lambda_*>0$, and an open interval $I_{\lambda_*}\subset \R_+$ containing $\lambda_*$ such that
			\[
			2\phi_k^\prime(0_+;\lambda)+k^2\alpha=0
			\]
			if and only if $(k,\lambda)=(k_*,\lambda_*)$ for any $k\in \N$ and $\lambda \in I_{\lambda_*}$.
		\end{itemize}
	}
\end{rem}

In Section  \ref{S:Gamma_d} we prove Theorem \ref{main2}. 
Claiming periodicity in one spatial direction and evenness in the transverse direction we make a Fourier ansatz of the form
\begin{equation}\label{eq:form_a}
	\Phi(x,y;\lambda)=\sum_{k\in \N}a_k(\lambda)\phi_k({|y|};\lambda)\sin(kx),
\end{equation}
where $\phi_k$ are the decaying functions from assumption $(\tilde L3)$.   In particular, the ansatz \eqref{eq:form_a} solves the linear differential equation in \eqref{eq:gleichung1} and thus reduces the problem of finding nontrivial solutions of \eqref{eq:Q_trav} to the following family of \emph{algebraic} equations
\begin{equation}\label{eq:R_a}
	2a_k(\lambda)\phi_k^\prime(0_+;\lambda)=-k^2\left( \alpha a_k(\lambda)-\frac{1}{4}\gamma \bigl(a(\lambda)*a(\lambda)*a(\lambda)\bigr)_k\right)\qquad\mbox{for all}\quad k\in \N.
\end{equation}

In the spirit of Section \ref{S:Gamma_b}, we show that if conditions $(\tilde L0)-(\tilde L3)$ are satisfied, then the nonlinear equation \eqref{eq:R_a} has a non-trivial solution $(a_k(\lambda))_{k\in \Z}$ with $a_k(\lambda)=-a_{-k}(\lambda)$ and the decay property
\[
(k^2a_k(\lambda))_k\in l^2(\R);
\]
thereby providing a solution of \eqref{eq:Q_trav} in the form of \eqref{eq:form_a}. 
Here, $l^2(\R)$ denotes the space of $l^2$--summable sequences in $\R^{\Z}$. 
Condition $(\tilde L1)$ guarantees the existence of the family of even ansatz-functions $(\phi)_{k\in \N}$ with $\phi_k(0;\lambda)=1$. Condition $(\tilde L2)$ assures that there exists $k_*\in \N$, $\lambda_*\in\R$ and an interval $I_{\lambda_*}\subset \R$ including $\lambda_*$ such that the linearization of \eqref{eq:R_a}, given by the multiplication operator 
\[
A_\lambda^k:= 2\phi_k^\prime(0_+;\lambda)+k^2\alpha,
\]
has a one-dimensional kernel if $(k,\lambda)=(k_*,\lambda_*)$, that is $A_{\lambda_*}^{k*}=0$; and $A_\lambda^k\neq 0$ for all $(k,\lambda)\in \N\times I_{\lambda_*}$ with $(k,\lambda)\neq(k_*,\lambda_*)$. This is a necessary bifurcation condition. After we have proved Theorem \ref{main2}, we investigate the specific cases, when $V$ is a potential of the form as in (P1) and (P2). The former being a $\delta$-potential on a constant background, while the latter is a $\delta$-potential on the background of a step function. In both cases the $\delta$-potential part in $V$ is essential, guaranteeing sufficient decay properties of the sequence $(a_k(\lambda))_{k\in \N}$.
 In particular, we prove the following corollaries:

\begin{cor}[Case P1, distributional $\Gamma$]  
\label{cor:p1_distrib}
Assume that $\Gamma=\gamma \delta_0$ and $V(\lambda,y)=\lambda +\alpha \delta_0(y)$. If $k_*\in \N$ and $\lambda_*<1$ are given and $\alpha>0$ is determined from 
\begin{equation*}\label{eq:bfp1_distrib}
\alpha =\frac{2\sqrt{1-\lambda_*}}{k_*} 
\end{equation*}
then the assumptions in Theorem~\ref{main2} are satisfied with
\[
\phi^* (y)=e^{-k_*\sqrt{1-\lambda_*}|y|}
\]
and
\begin{align*}
	\Phi(0)&=0,  && D_\e \Phi(0)(x,y)=e^{-k_*\sqrt{1-\lambda_*}|y|}\sin(k_*x), \\
	\lambda(0)&=\lambda_*, && \dot \lambda(0)=0, \quad \ddot \lambda(0)= -\gamma k_* \sqrt{1-\lambda_*}.
\end{align*}
Moreover, the solutions $\Phi(\e)$ take the form
  \[
  \Phi(\e)(x,y)=\sum_{k\in \N}a_k(\e)e^{-k\sqrt{1-\lambda}|y|}\sin(kx).
  \]
\end{cor}

\begin{cor}[Case P2, distributional $\Gamma$]  
	\label{cor:p2_distrib}
	Let $\Gamma \in L^\infty(\R)$ and $V(\lambda, y)=\lambda \textbf{1}_{|y|\geq b}+ \beta  \textbf{1}_{|y|<b}+ \alpha \delta_0(y)$. Suppose furthermore that $k_*\in \N$ and $\lambda_*<1$ are given. If
\begin{itemize}
\item (Case $\beta>1$) $b, \alpha>0$ are determined from
    \[
	b=\frac{\pi}{\sqrt{\beta-1}}, \quad \alpha=\frac{2\sqrt{1-\lambda_*}}{k_*}
	\]
\item (Case $\beta<1$) $b>0$ is given and $\alpha>0$ is determined from
    \[
    \alpha =\frac{2\sqrt{1-\beta}}{k^*}\cdot \frac{\sqrt{1-\beta}\sinh(k_*\sqrt{1-\beta}b)+\sqrt{1-\lambda_*}\cosh(k_*\sqrt{1-\beta}b)}{\sqrt{1-\beta}\cosh(k_*\sqrt{1-\beta}b)+\sqrt{1-\lambda_*}\sinh(k_*\sqrt{1-\beta}b)} 
    \]
\item (Case $\beta=1$) $b>0$ is given and $\alpha>0$ is determined from
    \[
    \alpha = \frac{2\sqrt{1-\lambda_*}}{k_*(1+\sqrt{1-\lambda_*}k_*b)}
    \]
\end{itemize}
then in all three cases the assumptions in Theorem~\ref{main2} are satisfied. 
\end{cor}

\begin{rem} 
	\emph{Details on the construction of $\phi^*$ in Corollary \ref{cor:p2_distrib}, the functions $\phi_k$  and the form of the solutions $\Phi(\e)(x,y)=\sum_{k\in \N}a_k(\e)\phi_k(y;\lambda)\sin(kx)$ can be taken from Section~\ref{Ss:P2_distrib}.
	}
\end{rem}

   \bigskip
   
   \section{Notation}\label{S:notation}
    If $f$ and $g$ are elements in an ordered Banach space, we write $f\lesssim g$ ($f\gtrsim g$) if there exists a constant $c>0$ such that $f\leq c g$ ($f\geq cg$). Moreover, the notation $f\eqsim g$ is used whenever $f\lesssim g$ and $f\gtrsim g$. We write $c=c(p_1,p_2,\ldots)>0$  if we want to emphasize  that the constant $c>0$ depends on the parameters $p_1,p_2,\ldots$.  In Section~\ref{S:Gamma_d} we are looking for solutions of an infinite dimensional system of nonlinear algebraic equation. We consider solutions in the sequence spaces related to 
   \[
   l^2(\R):=\left\{a=(a_k)_{k\in \Z}\mid a_k\in\R  \mbox{ for all } k\in \Z \mbox{ and } \|a\|_{l^2(\R)}^2:=\sum_{k\in\Z}|a_k|^2<\infty\right\}.
   \]
   Eventually, for any $r\in \R$ we set
   \[
   h^r(\R):=\left\{a\in l^2(\R)\mid  \bigl((1+|k|)^r a_k\bigr)_{k\in\Z} \in l^2(\R)\right\}.
   \]
    and equip the space $h^r(\R)$ with the norm
    \[
    \|a\|_{h^r(\R)}^2:=\sum_{k\in \Z}(1+|k|)^{2r}|a_k|^2.
    \]
   We also consider the subspaces
   \begin{align*}
   l^2_\sharp(\R) &:=\{a \in l^2(\R): a_{-k} = -a_k \mbox{ for } k \in \Z\}, \\
   h^r_\sharp (\R) &:= h^r(\R)\cap l^2_\sharp(\R).
   \end{align*}
   Throughout the paper we use the notation $\langle \cdot,\cdot \rangle_H$ to denote the dual pairing in the Hilbert space $H$. If $f,g\in L^2(U)$  are real-valued functions, where $U\subset \R^n$ is a domain in $\R^n$, $n\in \N$, then
   \[
   \langle f,g\rangle_{L^2(U)}:=\int_{U}f(z)g(z)\,dz
   \]
   and if $a,b\in l^2(\R)$ then
   \[
	   \langle a,b \rangle_{l^2(\R)}:=\sum_{k\in \Z}a_kb_k.
   \]
   
If $L:D(L)\subset H \to H$ is a linear operator with domain $D(L)$,  we denote by
 \[
 \rho(L):= \{\lambda \in \C \mid \lambda-L:D(L)\to H\; \mbox{has a bounded inverse}\}
 \]
 the \emph{resolvent set} of $L$. The \emph{spectrum} of $L$ is given by $\C \setminus \rho(L)$. If $L$ is self-adjoint, then $\sigma(L)\subset \R$ and the spectrum of $L$ can be decomposed as a disjoint union
 \[
 \sigma(L)=\sigma_{ess}(L)\cup \sigma_d(L),
 \]
 where $\sigma_{d}$ is the \emph{discrete spectrum} of $L$ consisting  of isolated eigenvalues of $\sigma(L)$ of finite multiplicity and $\sigma_{ess}(L)=\sigma(L)\setminus \sigma_{d}(L)$ is the \emph{essential spectrum}.

\bigskip

	\section{Existence of  traveling waves for bounded potentials $\Gamma$}
\label{S:Gamma_b}

This section is devoted to the proof of Theorem~\ref{main1}.
Subsequently, we affirm in Section~\ref{S:regular} that the conditions are fulfilled for special cases where $V$ takes the form in (P1) or (P2), therey proving Corollary~\ref{cor:p1} and Corollary~\ref{cor:p2}.
In the following we restrict ourself to solutions $\Phi$ of \eqref{eq:Q_trav} having the form
\begin{equation*}
	\Phi(x,y;\lambda)=\sum_{k\in \N}\phi_k(y;\lambda)\sin(kx),
\end{equation*}
where the $y$-dependent Fourier coefficients $(\phi_k)_{k\in \N}$ are decaying at infinity (suitable function spaces are formulated later). Then, $\Phi$ is a solution of
\begin{equation}\label{eq:Eq}
	-\Phi_{yy}-(1-\lambda V_0(y)-V_1(y))\Phi_{xx}+\Gamma(y)(\Phi^3)_{xx}=0
\end{equation} 
if and only if
\[
-\phi_k^{\prime\prime}+k^2(1-\lambda V_0(y)-V_1(y))\phi_k+\frac{1}{4}k^2\Gamma(y)\left(\Phi*\Phi*\Phi\right)_k=0\qquad\mbox{for all}\quad k\in \N.
\]
Note that $-\frac{1}{4}\left(\Phi*\Phi*\Phi\right)$ is the $k$-th Fourier coefficient of $\Phi^3$, cf. Lemma~\ref{A1}. The Fourier ansatz with respect to $x$ decomposes the operator 
\begin{equation}\label{eq:L}
	L_\lambda:=-\partial_y^2 -(1-\lambda V_0(y)-V_1(y))\partial_x^2
\end{equation}
into the sequence of Schr\"odinger operators $L_\lambda^k$
\[
L_\lambda^k:=-\frac{d^2}{dy^2}+k^2(1-\lambda V_0(y)-V_1(y)).
\]
Recall that we are working under the assumptions $(L0)$--$(L3)$ from Theorem~\ref{main1}.
\begin{rem} \label{gap_and_embedding}
	\emph{
		\begin{itemize}
			\item[(i)] Notice that a necessary condition for  $(L3)$ to hold is that the operator $L^k_\lambda$ satisfies the spectral gap property
			\[
			(-ck^2,ck^2)\subset \rho(L^k_\lambda)\qquad\mbox{for some constant}\quad c>0
			\]
			uniformly in $\lambda\in I_{\lambda_*}$ and $k\in \N$ sufficiently large.\\
			\item[(ii)] The domain of $L_\lambda^k$ is a subset of $H^1(\R)$, which is the domain of the quadratic form of $L_\lambda^k$. As a vector space, it does not depend on $\lambda$. However, the graph norm on $D(L_\lambda^k)$ is $\lambda$-dependent and the embedding $D(L_\lambda^k)\subset H^1(\R)$ is locally uniformly bounded with respect to $\lambda$.
		\end{itemize}
	}
\end{rem}

The next lemma extends property $(L3)$ to all values of $k\in \N$ by adding a projection to $L_\lambda^k$ for $k=k_\ast$. For this purpose let $\ker L_{\lambda_*}^{k_*}=\spa\{\phi^*\}$ with $\|\phi^*\|_{L^2(\R)}=1$. Denote by $P^{k_*}$ the projection mapping
\[
P^{k_*}\phi:=\langle \phi, \phi^*\rangle_{L^2(\R)} \phi^*\qquad \mbox{for any}\quad \phi \in L^2(\R)
\]  
and define by $\tilde L_\lambda^k:D(L_\lambda^k)\subset L^2(\R)\to L^2(\R)$ for $k\in \N$ and $\lambda\in I_{\lambda_*}$ the family of operators
\begin{equation} \label{def_projectedop}
	\tilde L^{k}_{\lambda}=\left\{ \begin{array}{lcl}L^{k_*}_\lambda +P^{k_*}\quad&\mbox{if}&\quad k=k_*, \vspace{\jot} \\
		L^k_\lambda\quad&\mbox{if}&\quad k\neq k_*. \end{array}\right.
\end{equation}

\begin{lem} \label{lem:inverse} Let $\tilde L_\lambda^k \phi =g$ for some $g\in L^2(\R)$. Then, by possibly shrinking the interval $I_{\lambda_*}$, we have that 
	\begin{equation} \label{eq:est_k}
		\|\phi\|_{L^2(\R)}\lesssim \frac{1}{k^2}\|g\|_{L^2(\R)}\qquad\mbox{and}\qquad	\|\phi^\prime \|_{L^2(\R)}\lesssim \frac{1}{k}\|g\|_{L^2(\R)}
	\end{equation}
	for all $k\in \N$ uniformly in $\lambda\in I_{\lambda_*}$.
\end{lem}

\begin{proof} By Theorem~VIII.25 in \cite{reed_simon} it follows that the map $\lambda \mapsto L_\lambda^k$ is norm-resolvent continuous, that is $\lambda \mapsto (L_\lambda^k-\mathrm i)^{-1}\in {\mathcal{L}}(L^2(\R))$ is continuous with respect to the operator norm. Let us verify that also $\tilde L_\lambda^{k*}$ is norm-resolvent convergent to $\tilde L_{\lambda_*}^{k*}$ as $\lambda\to \lambda_*$. Note that $\Id+P^{k_*}(L_{\lambda_*}^{k_*}-\mathrm i)^{-1}: L^2(\R) \to L^2(\R)$ is a compact perturbation of the identity, injective and hence bijective. Then, for $\lambda$ close to $\lambda_*$, also $\Id+P^{k_*}(L_\lambda^{k_*}-\mathrm i)^{-1}: L^2(\R) \to L^2(\R)$ is bijective. Note that we have the identity
	\[
	\left(L_\lambda^{k_*}-\mathrm i+ P^{k_*}\right)^{-1}=\left( \left( \Id + P^{k_*}(L_\lambda^{k_*}-\mathrm i)^{-1} \right)(L_\lambda^{k_*}-\mathrm i)\right)^{-1}.
	\]
	From this we see that
	\[
	(\tilde L_\lambda^{k_*}-\mathrm i)^{-1}=(L_\lambda^{k_*}-\mathrm i)^{-1}\left( \Id + P^{k_*}(L_\lambda^{k_*}-\mathrm i)^{-1} \right)^{-1}.
	\]
	Using the assumption that $L^{k_*}_\lambda$ converges to $L^{k_*}_{\lambda_*}$ in the norm resolvent sense, this implies the claim.

	Next we show that $0\in \rho(\tilde L^{k_*}_{\lambda_*})$. Since adding a (compact) projection operator only changes the discrete spectrum, we may assume by (L2) for contradiction that $0$ is an eigenvalue of $\tilde L^{k_*}_{\lambda_*}$, that is $L_{\lambda_*}^{k_*}\phi + P^{k_*}\phi=0$. Testing with $\phi^*$, which spans the kernel of $L_{\lambda_*}^{k_*}$, we get $\langle \phi, \phi^*\rangle_{L^2(\R)}=0$ and hence $P^{k_*}\phi=0$. Thus, $\phi$ also belongs to the kernel of $L_{\lambda_*}^{k_*}$, which contradicts $\langle \phi, \phi^*\rangle_{L^2(\R)}=0$ and the simplicity of the $0$-eigenvalue of $L_{\lambda_*}^{k_*}$.

	Finally, by  $(L3)$ weknow that there exists $k_0\in \N$ (we assume w.l.o.g. $k_0\geq k_*$) such that \eqref{eq:est_k} holds for $k> k_0\geq k_*$. This implies that 
	\begin{equation} \label{inf_dist}
		\inf_{k>k_0}\inf_{|\lambda-\lambda_*|<\delta} \dist(0, \sigma(\tilde L_\lambda^k))>0.
	\end{equation}
	Now we want to extend this inequality to the remaining  values of $k\in\{1,\ldots,k_0\}$ by possibly diminishing $\delta$. Thus, let $k\in \N$ with $1\leq k \leq k_0$ and assume for contradiction the existence of a sequence $\lambda_n\to \lambda_*$ as $n\to\infty$ such that there exists $\mu_n\in \sigma(\tilde L_{\lambda_n}^k)$ with $\lim_{n\to\infty} \mu_n=0$. By norm-resolvent convergence this implies $0\in \sigma(\tilde L_{\lambda_*}^k)$, which is impossible for $k\neq k_*$ by $(L2)$ and also impossible for $k=k_*$ as stated above. This contradiction establishes \eqref{inf_dist} for all $k\in \N$. Finally, \eqref{inf_dist} shows that the map $\lambda \mapsto \|(\tilde L_\lambda^k)^{-1}\|_{L^2(\R)\to L^2(\R)}$ is bounded for $\lambda \in (\lambda_*-\delta,\lambda_*+\delta)$ uniformly for $k\in \N$. The same holds for $\lambda \mapsto \|(\tilde L_\lambda^k)^{-1}\|_{L^2(\R)\to D(L_\lambda^k)}$, and due to (ii) in Remark~\ref{gap_and_embedding}, also for $\lambda \mapsto \|(\tilde L_\lambda^k)^{-1}\|_{L^2(\R)\to H^1(\R)}$. This establishes \eqref{eq:est_k} for all $k\in \N$.
\end{proof}

Now, we introduce suitable function spaces and use Lemma \ref{lem:inverse} to reformulate the nonlinear problem \eqref{eq:Eq} in a setting, which makes the local bifurcation theorem due to Crandall--Rabinowitz \cite{CrRab_bifurcation} applicable. Set
\[
X:=\left\{\Phi\in H^2(\T;L^2(\R))\cap H^1(\T;H^1(\R))\mid \Phi(x,y)=\sum_{k\in \N}\phi_k(y)\sin(kx) \right\}
\] 
and
\[
Y:=\left\{\Phi \in L^2(\T; L^2(\R)) \mid \Phi(x,y)=\sum_{k\in \N}\phi_k(y)\sin(kx)\right\}.
\]
Moreover,
we set
\[
\tilde L_\lambda:= L_\lambda +P^*,
\]
where $P^*$ denotes the  $L^2$-orthogonal projection onto $\ker L_{\lambda_*}=\spa\{\Phi^*\}$ with $\Phi^*(x,y)=\frac{1}{\sqrt{\pi}}\phi^*(y)\sin(k_* x)$. Recall, that the operator $L_\lambda$ is defined in \eqref{eq:L}. As an immediate consequence of $(L0)$--$(L3)$, we obtain the following lemma.

\begin{lem}\label{lem:1D}
	Assume that $(L0)$--$(L3)$ holds true and let $\lambda\in I_{\lambda_*}$. There exists a bounded linear map $\tilde L_\lambda^{-1}: Y \to X$ with the following property: if $f\in Y$ is given and the function $\Phi\in X$ solves
	\begin{equation} \label{eq:inv}
		\Phi = \tilde L_\lambda^{-1}(f + P^*\Phi)
	\end{equation}
	then $\Phi$ solves
	\begin{equation*} \label{eq:direkt}
		L_\lambda \Phi = f
	\end{equation*}
	in the weak sense, that is
	\[
	\int_{\T}\int_{\R} \Phi_y\Psi_y - \left(1-\lambda V_0(y)\right)\Phi_{xx}\Psi \,dy\,dx - \int_\T \langle V_1(\cdot)\Phi_x(x,\cdot), \Psi_x(x,\cdot) \rangle\,dx=\int_{\T}\int_{\R} f\Psi \,dy\,dx
	\]
	for any $\Psi \in X$.
\end{lem}	

\begin{proof}
	Let $\lambda \in I_{\lambda_*}$. For $g\in Y$ the definition of $\Phi := \tilde L_\lambda^{-1}g$ is given by 
	$$
	\Phi(x,y) = \sum_{k\in \N} \phi_k(y)\sin(kx) \quad \mbox{ with } \quad \phi_k = (\tilde L_\lambda^k)^{-1}g_k.
	$$
	Then Lemma~\ref{lem:inverse} implies that $\Phi\in X$ and that $\tilde L_\lambda^{-1}:Y\to X$ is bounded. Now suppose that $\Phi\in X$ solves \eqref{eq:inv}. Then $\phi_k= (\tilde L_\lambda^k)^{-1}(f_k + P^{k_*}\phi_k)$ so that $\phi_k\in D(\tilde L_\lambda^k)=D(L_\lambda^k)$ for all $k\in \N$. In particular, we know that 
	$$
	L_\lambda^k \phi_k = f_k \qquad \mbox{ for all } \qquad k\in \N
	$$
	and thus for $K_0\in \N$ and $\psi_1,\ldots, \psi_{K_0}\in H^1(\R)$ we have  
	$$
	\sum_{k=1}^{K_0}\left(\int_\R \phi_k' \psi_k' + k^2(1-\lambda V_0(y)\,dy- \langle V_1\phi_k,\psi_k\rangle\right)= \sum_{k=1}^{K_0} \int_\R f_k\psi_k\,dy.  
	$$
	Taking the limit $K_0\to\infty$ in the previous equation will lead to 
	$$
	\int_{\T}\int_{\R} \Phi_y\Psi_y - \left(1-\lambda V_0(y)\right)\Phi_{xx}\Psi\,dy\,dx - \int_\T \langle V_1(\cdot)\Phi_x(x,\cdot), \Psi_x(x,\cdot)\rangle\,dx=\int_{\T}\int_{\R} f\Psi \,dy\,dx
	$$
	for any $\Psi \in X$ due to the following estimates:
	\begin{align*}
		\int_{\T}\int_{\R} |\Phi_y\Psi_y| \,dy\,dx & \leq \sum_{k\in \N} \|\phi_k'\|_{L^2(\R)} \|\psi_k'\|_{L^2(\R)} \leq \|\Phi\|_{L^2(\T,H^1(\R))} \|\Psi\|_{L^2(\T,H^1(\R))}, \\
		\int_{\T}\int_{\R} |V_0(y) \Phi_{xx} \Psi| \,dy\,dx & \leq \sum_{k\in \N} \|V_0\|_{L^\infty(\R)} k^2\|\phi_k\|_{L^2(\R)} \|\psi_k\|_{L^2(\R)} \\ &  \leq \|V_0\|_{L^\infty(\R)}  \|\Phi\|_{H^2(\T,L^2(\R))} \|\Psi\|_{L^2(\T,L^2(\R))}, \\
		\int_\T |\langle V_1(\cdot)\Phi_x(x,\cdot), \Psi_x(x,\cdot)\rangle|\,dx & \leq \sum_{k\in \N} \|V_1\|_{H^1\to H^{-1}} \|k\phi_k\|_{H^1(\R)} \|k\psi_k\|_{H^1(\R)}\\  & \leq \|V_1\|_{H^1\to H^{-1}} \|\Phi\|_{H^1(\T,H^1(\R))} \|\Psi\|_{H^1(\T,H^1(\R))}, \\
		\int_{\T}\int_{\R} |f \Psi| \,dy\,dx & \leq \sum_{k\in \N} \|f_k\|_{L^2(\R)} \|\psi_k\|_{L^2(\R)} \leq \|f\|_{L^2(\T,L^2(\R))} \|\Psi\|_{L^2(\T,L^2(\R))}.
	\end{align*}
\end{proof}

Equipped with the above lemma, we use the invertibility of $\tilde L_\lambda$ to reformulate \eqref{eq:Eq} as
\begin{equation}\label{eq:F}
	F(\Phi,\lambda)=0,
\end{equation}
where $F: X \times I_{\lambda_*}\to X$ is given by
\[
F(\Phi,\lambda)=\Phi+\tilde L_\lambda^{-1}\left(\Gamma(y)(\Phi^3)_{xx}-P^*\Phi \right).
\]
We want  to apply bifurcation theory to equation \eqref{eq:F}.
Clearly, $F(0,\lambda)=0$ for any $\lambda\in I_{\lambda_*}$ and the line $\{(0,\lambda)\mid \lambda \in I_{\lambda_*}\}$ constitutes the line of trivial solutions from which we aim to bifurcate at $\lambda =\lambda_*$. The following lemma collects the necessary properties of the map $F$. 

\begin{lem} \label{lem:prop_F}
	The map $F:X\times I_{\lambda_*} \to X$ is a $C^\infty$-map. Moreover the following holds:
	\begin{itemize}
		\item [(i)] The linearization of $F$ about $\Phi=0$, given by
		\[
		D_\Phi F(0,\lambda)= \Id - \tilde L_\lambda^{-1}P^*:X \to X
		\]
		is a Fredholm operator of index zero.  In particular $D_\Phi F(0,\lambda_*)=\Id-P^*$. The kernel of $D_\Phi F(0,\lambda)$ is trivial for $\lambda \in I_{\lambda_*}\setminus\{\lambda_*\}$ and it is given by $\spa\{\Phi^*\}$ if $\lambda=\lambda_*$.
		\item[(ii)] The mixed second derivative of $F$ about $\Phi=0$ is given by 
		$$
		D^2_{\Phi,\lambda} F(0,\lambda)= \tilde L_\lambda^{-1} V_0\partial_x^2 \tilde L_\lambda^{-1} P^*: X\to X.
		$$
	\end{itemize}
\end{lem}

\begin{proof} Let us first verify the mapping properties of $F$ by checking that $(\Phi^3)_{xx}\in Y$ for $\Phi\in X$. First note that $\Phi\in X$ implies 
	\begin{align}\label{infinity}
		\begin{split}
			\|\Phi\|_\infty &= \sup_{x\in \R,y\in \T} \Bigl|\sum_{k\in \N} \phi_k(y)\sin(kx)\Bigr| \leq \sum_{k\in \N} \|\phi_k\|_{L^\infty(\R)} \leq \sum_{k\in\N} \|\phi_k\|_{H^1(\R)} \\
			& \leq \Bigl(\sum_{k\in \N} \frac{1}{k^2}\Bigr)^\frac{1}{2} \Bigl(\sum_{k\in \N} k^2 \|\phi_k\|_{H^1(\R)}^2\Bigr)^\frac{1}{2}\lesssim\|\Phi\|_{H^1(\T,H^1(\R))} 
		\end{split}
	\end{align}
	and, using $|\cos(kx)|\leq 1$,
	\begin{align*}
		\Bigl(\int_\T\int_\R |\Phi_x|^4 \,dy\,dx\Bigr)^{1/4}  & \leq \sum_{k\in \N} \|k \phi_k\|_{L^4(\T\times\R)} = \sqrt[4]{2\pi} \sum_{k\in \N} |k|\|\phi_k\|_{L^4(\R)}  \lesssim \sum_{k\in \N} |k| \|\phi_k'\|_{L^2(\R)}^\frac{1}{4} \|\phi_k\|_{L^2(\R)}^\frac{3}{4},
	\end{align*}
	by the Gagliardo--Nirenberg inequality, cf. \cite{Friedman}. Using a triple H\"older inequality we obtain that
	\begin{align}\label{l4}
		\begin{split}
			\Bigl(\int_\T\int_\R |\Phi_x|^4 \,dy\,dx\Bigr)^\frac{1}{4} & \leq C\Bigl(\sum_{k\in \N} k^2 \|\phi_k'\|_{L^2(\R)}^2\Bigr)^\frac{1}{8} \Bigl(\sum_{k\in \N} k^4 \|\phi_k\|_{L^2}^2\Bigr)^\frac{3}{8} \Bigl(\sum_{k\in \N} k^{-3/2}\Bigr)^\frac{1}{2}\\
			&   \lesssim \|\Phi\|_{H^1(\T,H^1(\R))}^\frac{1}{4} \|\Phi\|_{H^2(\T,L^2(\R))}^\frac{3}{4}.
		\end{split}
	\end{align}
	Hence, for $\Phi\in X$ we have $(\Phi^3)_{xx} = 3\Phi^2 \Phi_{xx} + 6\Phi\Phi_{x}^2 \in L^2(\T,L^2(\R))\subset Y$ and thus the mapping properties of $F$ are proved.

	The differentiability properties of $F$ with respect to $\Phi$ also follow in a similar way from $\Phi\in L^\infty(\T\times \R)$ and $\Phi_x\in L^4(\T\times \R)$. This can be seen as follows: the (formal) first/second derivatives of $F$ with respect to $\Phi$ are linear/bilinear operators and contain terms of the form $abc_{xx}$ or $ab_xc_x$ where $a,b,c\in X$. Based on the estimates
	\begin{align*}
		\int_\T\int_\R |a b_xc_x|^2\,dy\,dx & \leq \|a\|_\infty^2 \|b_x\|_{L^4(\T \times \R)}^2 \|c_x\|_{L^4(\T \times \R)}^2,\\
		\int_\T\int_\R |abc_{xx}|^2\,dy\,dx & \leq \|a\|_\infty^2 \|b\|_\infty^2 \|c_{xx}\|_{L^2(\T \times \R)}^2
	\end{align*}
	we find in view of \eqref{infinity} and \eqref{l4} that the first/second derivatives of $F$ with respect to $\Phi$ exist, are bounded linear/bilinear operators from $X$ to $Y$, and depend continuously on $\Phi$ and $\lambda$. Due to the cubic nature of the nonlinearity, derivatives of $F$ of order higher than two with respect to $\Phi$ are independent of $\Phi$.

	The differentiability properties of $F$ with respect to $\lambda$ follow from 
	\begin{equation} \label{eq:diff_lambda}
		\frac{d}{d\lambda} \tilde L_\lambda^{-1} = -\tilde L_\lambda^{-1} \frac{d}{d\lambda} \tilde L_\lambda \tilde L_\lambda^{-1} = -\tilde L_\lambda^{-1} V_0 \partial_x^2 \tilde L_\lambda^{-1}
	\end{equation}
	and due to $\tilde L_\lambda^{-1}: Y\to X$ and $V_0\partial_x^2: X\to Y$, we see that the resulting operator on the right-hand side of \eqref{eq:diff_lambda} is indeed a bounded linear map from $Y\to X$. Moreover, \eqref{eq:diff_lambda} explains the formula for $D^2_{\Phi,\lambda} F(0,\lambda)$ in (ii).

	Finally, the formula in (i) shows that $D_\Phi F(0,\lambda)$ is a compact perturbation of the identity, and hence Fredholm of index zero. Let us compute the kernel of $D_\Phi F(0,\lambda)$. If $\Psi\in X$ satisfies $D_\Phi F(0,\lambda)\Psi=0$ then according to Lemma~\ref{lem:1D} we have that $\Psi$ is a weak solution of $L_\lambda\Psi=0$. Then, for $\lambda\neq\lambda_*$ we have $\psi_k=0$ for all $k\in \N$ and hence $\Psi=0$. For $\lambda=\lambda_*$ we have $\psi_k=0$ for all $k\in \N\setminus\{k_*\}$ and $\psi_{k_*}\in \spa\{\phi^*\}$ so that $\Psi\in \spa\{\Phi^*\}$ as claimed.  Notice finally that $D_\Phi F(0,\lambda_*)=\Id-\tilde L_{\lambda_*}^{-1}P^*=\Id-P^*$ since $\range P^*=\spa\{\Phi^*\}$ is the eigenspace of $\tilde L_{\lambda*}$ corresponding to the eigenvalue $1$. This finishes the proof. 
\end{proof}

\begin{rem} \label{rem:trouble}\emph{
Let us briefly describe the difficulty that arises when one considers the bifurcation problem for $V(\lambda,y)=\lambda V(y)$ with $V=V_0+V_1$, i.e., when multiplication with the bifurcation parameter is extended to the distributional potential $V_1$. In this case one already obtains a problem in verifying the $C^1$-property of the map $F$. Formally one finds 
$$
D_\lambda F(\Phi,\lambda) = -\tilde L_\lambda^{-1} V \partial_x^2 \tilde L_\lambda^{-1}(\Gamma(y)(\Phi^3)_{xx}-P^*\Phi).
$$
As above, we would expect to have $\tilde L_\lambda^{-1} V \partial_x^2 \tilde L_\lambda^{-1}:Y\to X$ as a bounded linear map. But this is not the case, as a calculation in the case where $V_0(y)\equiv 1$ and $V_1(y)=\alpha\delta_0(y)$ shows. Namely, let $A= \tilde L_\lambda^{-1}:Y\to X$ and $B=V \partial_x^2$. Then $B:X\to H^{-1}(\T;H^{-1}(\R))$ and $C=\tilde L_\lambda^{-1}: \range(B)\to H^{3/2}(\T;L^2(\R))\cap H^{1/2}(\T;H^1(\R))\not\subset X$, i.e., we are missing a half-derivative in the regularity gain.
}
\end{rem}

The advantage of formulating the problem \eqref{eq:Eq} as $F(\Phi,\lambda)=0$ relies on the fact that  its linearization about $\Phi=0$ is of the form identity plus compact operator, which provides the Fredholm property for free. Applying the Crandall--Rabinowitz theorem (cf. e.g. \cite{CrRab_bifurcation} or \cite[Theorem I.5.1]{Kielhoefer}), we prove that assumption $(L0)$--$(L3)$ on the family of Schr\"odinger operators $L_\lambda^k$ are sufficient to guarantee the existence of nontrivial small-amplitude solutions of  \eqref{eq:Eq} provided a certain \emph{transversality condition} is satisfied, which we can formulate in terms of the potential $V_0$, see \eqref{eq:t}.

\medspace

\textbf{Proof of Theorem \ref{main1}.
}
    Recall from Lemma~\ref{lem:prop_F} that $D_\Phi F(0,\lambda_*)=\Id-P^*:X\to X$. Moreover, $D_\Phi F(0,\lambda_*)$ is a Fredholm operator of index zero with a one-dimensional kernel spanned by $\Phi^*$. Correspondingly, we can split the underlying space as follows: 
    $$
    X=\spa\{\Phi^*\}\oplus\spa\{\Phi^*\}^{\perp_{L^2}}=\ker(D_\Phi F(0,\lambda_*))\oplus\range(D_\Phi F(0,\lambda_*)).
    $$ 
    Hence, according to the Crandall--Rabinowitz theorem, the existence of a local bifurcation branch of nontrivial solutions  \eqref{eq:Q_trav} follows provided that the transversality condition
	\[
	D_{\Phi \lambda}^2 F(0,\lambda_*)\Phi^* \notin \range D_\Phi F(0,\lambda_*)=\spa\{\Phi^*\}^{\perp_{L^2}}
	\]
	is satisfied.  In view of $\tilde L_{\lambda_*}^{-1}\Phi^*=\Phi^*$ and the symmetry of $\tilde L_{\lambda_*}^{-1}$ (which follows from the  self-adjointness of $L^{k}_{\lambda_*}, \tilde L^{k}_{\lambda_*}$) the transversality condition holds since  
	$$
	\langle D_{\Phi \lambda}^2 F(0,\lambda_*)\Phi^*,\Phi^*\rangle_{L^2(\T\times\R)} =  \langle \tilde L_{\lambda_*}^{-1} V_0 \Phi^*_{xx},\Phi^*\rangle_{L^2(\T\times\R)} = \langle V_0 \Phi^*_{xx}, \Phi^* \rangle_{L^2(\T\times\R)} = -\pi k^2 \langle V_0 \phi^*_x,\phi^*_x\rangle_{L^2(\R)}\neq 0
	$$
	due to assumption \eqref{eq:t} of the theorem.

\medspace

Now, we are going to state the bifurcation formulas with the help of the Lyapunov--Schmidt reduction (cf. \cite[Theorem I.2.3]{Kielhoefer}). 
The Lyapunov--Schmidt reduction theorem in our context reads as follows:

\begin{thm}[Lyapunov--Schmidt reduction, \cite{Kielhoefer}, Theorem I.2.3]  Let $F:X\times I\to X$ be a $C^\infty$-map and $X=N\oplus N^{\perp_{L^2}}$ with $N=\spa \{\Phi^*\}=\ker D_\Phi F(0,\lambda_*)$ and $\lambda_*\in I$.
	There exists a neighborhood $O\times I'\subset \{(\Phi,\lambda)\in X\times \R_+\} $ of the bifurcation point $(0,\lambda_*)$ such that the problem
	\[
	F(\Phi,\lambda)=0\qquad \mbox{for}\quad (\Phi,\lambda)\in O\times I'
	\]
	is equivalent to the finite-dimensional problem
	\begin{equation}\label{eq:LF}
		\eta(\e \Phi^*,\lambda):=P^*F(\e \Phi^*+\psi(\e\Phi^*,\lambda),\lambda)=0
	\end{equation}
	for functions $\eta\in C^\infty(O_N\times I',N)$, $\psi \in C^\infty(O_N\times  I', N^{\perp_{L^2}})$ where $O_N\subset N$  is an open neighborhood of the zero element in $N$. One has that
	\[
	\eta(0,\lambda_*)=\psi(0,\lambda_*)=D_\Phi \psi(0,\lambda_*)=0
	\]
	and solving \eqref{eq:LF} provides a solution
	\[
	\Phi=\e\Phi^*+\psi(\e \Phi^*,\lambda)
	\]
	of the infinite-dimensional problem $F(\Phi,\lambda)=0$.
\end{thm}
We have the following Fr\'echet derivatives:
\begin{align*}
	D_\Phi F(\Phi,\lambda)\Phi^*&=\Phi^*+\tilde L_\lambda^{-1}\left(\Gamma(y)3(\Phi^2\Phi^*)_{xx}-P^*\Phi^*\right),\\
	D_{\Phi\Phi}^2 F(\Phi,\lambda)[\Phi^*,\Phi^*]&=\tilde L_\lambda^{-1}\left(\Gamma(y)6\Phi (\Phi^*)^2\right)_{xx},\\
	D_{\Phi\Phi\Phi}^3F(\Phi,\lambda)[\Phi^*,\Phi^*,\Phi^*]&=\tilde L_\lambda^{-1}\left(\Gamma(y)6(\Phi^*)^3\right)_{xx}.
\end{align*}

According to \cite[Section I.6]{Kielhoefer}, we have that
\[
\dot \lambda(0)=-\frac{1}{2}\frac{\langle D_{\Phi\Phi}^2F(0,\lambda_*)[\Phi^*,\Phi^*],\Phi^*\rangle_{L^2(\T\times\R)}}{\langle D_{\Phi \lambda}^2F(0,\lambda_*)\Phi^*,\Phi^*\rangle_{L^2(\T\times\R)}}.
\]
In view of $F$ being cubic in $\Phi$ it is clear that $\dot \lambda(0)=0$. In this case the second derivative is given by
\begin{equation}\label{eq:D}
	\ddot \lambda(0)=-\frac{1}{3}\frac{\langle D_{\Phi\Phi\Phi}^3\eta(0,\lambda_*)[\Phi^*,\Phi^*,\Phi^*],\Phi^* \rangle_{L^2(\T\times\R)}}{\langle D_{\Phi\lambda}^2F(0,\lambda_*)\Phi^*,\Phi^*\rangle_{L^2(\T\times\R)}}.
\end{equation}

\begin{prop}\label{prop:bf}
	Let $\{(\Phi(\e),\lambda(\e))\mid |\e|<\e_0\}\subset X\times I_{\lambda_*}$ be the local bifurcation curve found in Theorem \ref{main1} corresponding to the bifurcation point $(0,\lambda_*)$. Then
	\[
	\dot \lambda(0)=0\qquad \mbox{and}\qquad 	\ddot \lambda(0)= -\frac{3\pi}{2}\frac{\int_{\R}\Gamma(y)(\phi^*)^4(y)\,dy}{\int_\R V_0(y)(\phi^*)^2(y)\,dy}.
	\]
\end{prop}

\begin{proof}
	As already mentioned, the cubic nonlinearity of $F$ implies already that $\dot \lambda(0)=0$. We are left to compute the second derivative of $\lambda$ at the origin. According to the formula in \eqref{eq:D} we need to compute the third derivative of $\eta$ with  respect to $\Phi$ evaluated at $(0,\lambda_*)$. As for instance in  \cite[Eq. (I.6.5)]{Kielhoefer} we obtain that
	\begin{align*}
		D_{\Phi\Phi\Phi}^3\eta(0,\lambda_*)[\Phi^*,\Phi^*,\Phi^*]=&P^*D_{\Phi\Phi\Phi}^3F(0,\lambda_*)[\Phi^*,\Phi^*,\Phi^*]+3P^*D_{\Phi\Phi}^2F(0,\lambda_*)[\Phi^*,D_{\Phi\Phi}^2\psi(0,\lambda_*)[\Phi^*,\Phi^*]].
	\end{align*}
	Again, since $F$ is cubic in $\Phi$, we have that $D_{\Phi\Phi}^2F(0,\lambda_*)=0$, whence
	\begin{align*}
		\ddot \lambda(0)&=-\frac{1}{3}\frac{\langle P^*D_{\Phi\Phi\Phi}^3F(0,\lambda_*)[\Phi^*,\Phi^*,\Phi^*],\Phi^* \rangle_{L^2(\T\times\R)}}{\langle D_{\Phi\lambda}^2F(0,\lambda_*)\Phi^*,\Phi^*\rangle_{L^2(\T\times\R)}}\\
		&=-\frac{1}{3}\frac{\langle D_{\Phi\Phi\Phi}^3F(0,\lambda_*)[\Phi^*,\Phi^*,\Phi^*],\Phi^* \rangle_{L^2(\T\times\R)}}{\langle D_{\Phi\lambda}^2F(0,\lambda_*)\Phi^*,\Phi^*\rangle_{L^2(\T\times\R)}}.
	\end{align*}
	We have that
	\begin{align*}
		\langle D_{\Phi\Phi\Phi}^3F(0,\lambda_*)[\Phi^*,\Phi^*,\Phi^*],\Phi^* \rangle_{L^2(\T\times\R)}&=\langle \tilde L_{\lambda_*}^{-1}\left(\Gamma(y)6(\Phi^*)^3\right)_{xx},\Phi^* \rangle_{L^2(\T\times\R)}.
	\end{align*}
	Using the symmetry of $\tilde L_{\lambda_*}$ together with $\tilde L^{-1}_{\lambda_*}\Phi^*=\Phi^*$, we obtain that
	\begin{align*}
		\langle D_{\Phi\Phi\Phi}^3F(0,\lambda_*)[\Phi^*,\Phi^*,\Phi^*],\Phi^* \rangle_{L^2(\T\times\R)}&=6\langle \Gamma(y)\left((\Phi^*)^3\right)_{xx},\Phi^* \rangle_{L^2(\T\times\R)}=-\frac{9}{2}\pi k_*^2\int_{\R}\Gamma(y)(\phi^*)^4(y)\,dy
	\end{align*}
	and we know already that the denominator in $\ddot \lambda(0)$ is given by
	\begin{align*}
		\langle D_{\Phi \lambda}^2F(0,\lambda_*)\Phi^*,\Phi^*\rangle_{L^2(\T\times\R)} &= \langle V_0(y)\Phi^*_{xx},\Phi^*\rangle_{L^2(\T\times\R)}= -\pi k_*^2 \int_\R V_0(y)(\phi^*)^2(y)\,dy.
	\end{align*}
	Summarizing, we conclude that
	\[
	\ddot \lambda(0)= -\frac{3}{2}\frac{\int_{\R}\Gamma(y)(\phi^*)^4(y)\,dy}{\int_\R V_0(y)(\phi^*)^2(y)\,dy}.
	\]

\end{proof}


\section{Examples for regular $\Gamma$} \label{S:regular}

In what follows, we consider specific examples of potentials $V$ and prove Corollary~\ref{cor:p1} and Corollary~\ref{cor:p2}, which state the existence of traveling waves of \eqref{eq:Q_trav} in the specific case when the potentials are given as in (P1), (P2), respectively.  Both Corollary~\ref{cor:p1} and Corollary~\ref{cor:p2} are immediate consequences of Theorem~\ref{main1} and Proposition~\ref{prop:bf}, provided conditions $(L0)-(L3)$ are satisfied. 

\medskip

Recall that $V(\lambda,y)=\lambda V_0(y)+V_1(y)$ where in case (P1) we have $V_0(y)=1$, $V_1(y)=\alpha\delta_0(y)$ and in case (P2) we have $V_0(y)=\textbf{1}_{|y|\geq b}$, $V_1(y)=\beta  \textbf{1}_{|y|<b}+ \alpha \delta_0(y)$. The transversality condition \eqref{eq:t} is trivially satisfied, since in both cases $V_0\geq 0$ and $\not\equiv 0$. It is also clear that $(L0)$ holds true. The beginning of this section will be valid both for (P1) and (P2) since at the general level we may consider (P1) as a special case of (P2) with $\beta=\lambda$. In the subsequent subsections the considerations will split according to the two cases. 

\medskip

Let us consider the operator
\[
L_\lambda^k:=-\frac{d^2}{dy^2}+k^2(1-\lambda\textbf{1}_{|y|\geq b} -\beta  \textbf{1}_{|y|<b}-\alpha\delta_0(y))
\]
with $\lambda, \beta<1$. According to \cite{christ_stolz} the operator $L_\lambda^k:D(L_\lambda^k)\subset L^2(\R)\to L^2(\R)$ is self-adjoint on the domain
\[
D(L_\lambda^k)=\{\phi \in H^1(\R)\mid \phi\in H^2(-\infty,0)\cap H^2(0,\infty), \phi^\prime(0_+)-\phi^\prime(0_-)=-k^2\alpha\phi(0)\};
\]
thereby $(L1)$ is fulfilled. Moreover, $\sigma_{ess}(L_k)=[k^2(1-\lambda),\infty)$ according to Lemma \ref{lem:spectrum}. Next we consider the point spectrum of $L_k$, i.e., the eigenvalue problem of finding $\phi\in D(L_\lambda^k)$ with $L_\lambda^k \phi = k^2 \mu \phi$ where $\tilde\mu=k^2\mu$ is the actual eigenvalue. Setting $\tilde\lambda =\lambda+\mu$ and $\tilde\beta = \beta+\mu$ the eigenvalue problem then reduces to 
\begin{equation} \label{ev_problem}
\left\{\begin{split}
-\phi'' +k^2(1-\tilde\lambda\textbf{1}_{|y|\geq b}-\tilde\beta \textbf{1}_{|y|<b})\phi& =0, \quad y\in (-\infty, 0)\cup (0,\infty),\\
\phi'(0+)-\phi'(0-)+k^2\alpha \phi(0) &=0.
\end{split}\right.
\end{equation}
For reasons that will become obvious in the subsequent discussion we suppose $\mu$ to be so small that $\tilde\lambda, \tilde\beta<1$. In Lemma~\ref{eigenvalue_condition} in the Appendix we show that this problem is solvable (with a one-dimensional eigenspace) if and only if
\begin{equation}
\label{this_makes_an_eigenvalue}
\frac{k\alpha}{2\sqrt{1-\tilde\beta}} = \frac{\sqrt{1-\tilde\beta}\sinh(k\sqrt{1-\tilde\beta}b)+\sqrt{1-\tilde\lambda}\cosh(k\sqrt{1-\tilde\beta}b)}{\sqrt{1-\tilde\beta}\cosh(k\sqrt{1-\tilde\beta}b)+\sqrt{1-\tilde\lambda}\sinh(k\sqrt{1-\tilde\beta}b)}.
\end{equation}
Now we will split the discussion into subsections according to the cases (P1) and (P2), verifying $(L2)$ and $(L3)$ separately.

\subsection{(P1) $V$ a $\delta$-potential on a positive background}
\label{Ss:P1}

Here we take $V_0=1$ and $V_1=\alpha\delta_0$ and $V(\lambda,y)=\lambda + \alpha\delta_0(y)$ with $\alpha>0$ and $\lambda<1$.
In the subsequent results of Lemma~\ref{prop:A1}, Lemma \ref{prop:A2} we verify that the family of linear operators $L_\lambda^k$ satisfies also the assumptions $(L2)$ and $(L3)$ in Theorem \ref{main1}.  Since (P1) is a special case of (P2) with $\lambda=\beta$ we see that the eigenvalue condition \eqref{this_makes_an_eigenvalue} becomes 
\begin{equation} \label{ev_cond_p1}
\frac{k\alpha}{2\sqrt{1-\tilde\lambda}} =1.
\end{equation}
This leads to the following lemma.

\begin{lem}\label{prop:A1}
 Let us fix a wavenumber $k_*\in \N$ and let $\lambda_*<1$. We determine $\alpha>0$ such that
\begin{equation}\label{eq:bifurcation_condition}
    \alpha=\frac{2\sqrt{1-\lambda_*}}{k_*}.
\end{equation}
Then there exists an open interval $I_{\lambda_*}\subset \R$ containing $\lambda_*$ such that
	\[
	\dim \ker L_{\lambda_*}^{k_*}=1,
	\]
	and $\ker L_{\lambda}^k=\{0\}$ for any $(k,\lambda)\in \N\times I_{\lambda_*}$ with $(k,\lambda)\neq (k_*,\lambda_*)$.
\end{lem}

\begin{proof}
    Since we are considering the zero-eigenvalue of $L_{\lambda}^{k}$ we have $\mu=0$ and $\tilde\lambda=\lambda$. Together with our choice of $\alpha$ the eigenvalue condition \eqref{ev_cond_p1} becomes 
    $$
     \frac{k}{k_*} = \frac{\sqrt{1-\lambda}}{\sqrt{1-\lambda_*}}.
    $$
    Recall that $k\in \N$ is integer valued. Therefore, choosing a sufficiently small interval $I_{\lambda_*}\subset (-\infty,1)$ that contains $\lambda_*$ the eigenvalue condition is satisfied for $\lambda\in I_{\lambda_*}$ and $k\in \N$ if and only if $\lambda=\lambda_*$ and $k=k_*$. Moreover, for $k=k_*$ and $\lambda=\lambda_*$ the eigenspace is one-dimensional.
\end{proof}

 It is clear that kernel of $L_{\lambda_*}^{k_*}$ is spanned by the $L^2(\R)$-unitary element
\[
\phi^*(y):= \sqrt{k_*\sqrt{1-\lambda_*}}e^{-k_*\sqrt{1-\lambda_*}|y|}
\]
 since $L_{\lambda_*}^{k_*}\phi^*=0$ in $\R\setminus\{0\}$ and it satisfies ${\phi^*}'(0+)-{\phi^*}'(0-)+\alpha k_*^2\phi^*(0)=0$. 
The above lemma ensures that assumption $(L2)$ is satisfied.  The following lemma concerns the spectral properties of $ L_\lambda^k$ and shows that  assumption $(L3)$ holds true. 

\begin{lem}\label{prop:A2}
	There exists an open interval $I_{\lambda_*}\subset \R$ containing $\lambda_*$ such that the following holds for all $k\geq 3k_*$ and all $\lambda\in I_{\lambda_*}$:
	if $L_\lambda^k \phi =f$ for some $f\in L^2(\R)$, then
	\begin{equation*}\label{eq:Est}
		\|\phi\|_{L^2(\R)}\lesssim \frac{1}{k^2}\|f\|_{L^2(\R)}\qquad\mbox{and}\qquad	\|\phi^\prime \|_{L^2(\R)}\lesssim \frac{1}{k}\|f\|_{L^2(\R)}.
	\end{equation*}
	In particular, there exists a constant $c=c(k_*,|I_{\lambda_*}|)$, depending on $k_*$ and the size of the interval $I_{\lambda_*}$, such that 
	\begin{equation*}\label{eq:gap}
		(-ck^2,ck^2)\subset \rho (L_\lambda^k) \quad \mbox{ for every} \quad k \geq 3k_*, \lambda\in I_{\lambda_*}.
	\end{equation*} 
\end{lem}

\begin{proof} We show that for any $\lambda<1$ we have
	\begin{align*}
		\|L_\lambda^{k}\phi\|_{L^2(\R)}^2 \geq & \frac{1}{2}\left(\|\phi^{\prime\prime}\|_{L^2(-\infty,0)}^2+\|\phi^{\prime\prime}\|_{L^2(0,\infty)}^2\right)\\
		& +2(1-\lambda)\left(k^2-4k_*^2-\frac{16(\lambda_*-\lambda)}{\alpha^2}\right)\|\phi^{\prime}\|_{L^2(\R)}^2+\left(\frac{k_*^2\alpha^2}{4}+\lambda_*-\lambda\right)^2k^4\|\phi\|_{L^2(\R)}^2,
	\end{align*}		
	which proves the assertion. We have that
	\begin{align*}
		\|L_\lambda^{k}\phi\|_{L^2(\R)}^2 =  \int_{-\infty}^{0}(L_\lambda^{k}\phi)^2\, dy  +  \int_{0}^{\infty}(L_\lambda^{k}\phi)^2\, dy.
	\end{align*}
	For the first integral on the right hand side, we compute 
	\begin{align*}
		\int_{-\infty}^{0}(L_\lambda^{k}\phi)^2\, dy  &= \int_{-\infty}^{0}(-\phi^{\prime\prime}+k^2(1-\lambda)\phi)^2\,dy\\
		&= \int_{-\infty}^{0}|\phi^{\prime\prime}|^2 - 2k^2(1-\lambda)\phi^{\prime\prime}\phi +k^4(1-\lambda)^2\phi^2\,dy\\
		&=\|\phi^{\prime\prime}\|_{L^2(-\infty,0)}^2+2k^2(1-\lambda)\|\phi^{\prime}\|_{L^2(-\infty,0)}^2+k^4(1-\lambda)^2\|\phi\|_{L^2(-\infty,0)}^2 -2k^2(1-\lambda)\phi^\prime(0_-)\phi(0),
	\end{align*}
	where we used integration by parts. Similarly, we obtain that
	\begin{align*}
		\int_{0}^{\infty}(L_\lambda^{k}\phi)^2\, dy  &=\|\phi^{\prime\prime}\|_{L^2(0,\infty)}^2+2k^2(1-\lambda)\|\phi^{\prime}\|_{L^2(0,\infty)}^2+k^4(1-\lambda)^2\|\phi\|_{L^2(0,\infty)}^2+2k^2(1-\lambda)\phi^\prime(0_+)\phi(0).
	\end{align*}
	Taking the sum of the two integrals and using for $\phi\in D(L^k_\lambda)$ that 
	\begin{equation*}\label{eq:boundary}
		\phi^\prime(0_+)-\phi^\prime(0_-)=-k^2\alpha\phi(0)
	\end{equation*}
	we find that
	\begin{align}	\label{eq:ungleichung}
		\begin{split}
			\| L_\lambda^{k}\phi\|_{L^2(\R)}^2 = & \|\phi^{\prime\prime}\|_{L^2(-\infty,0)}^2+\|\phi^{\prime\prime}\|_{L^2(0,\infty)}^2+2k^2(1-\lambda)\|\phi^{\prime}\|_{L^2(\R)}^2+k^4(1-\lambda)^2\|\phi\|_{L^2(\R)}^2  \\
			& -\frac{2}{\alpha}(1-\lambda)\left(\phi^\prime(0_+)-\phi^\prime(0_-)\right)^2.
		\end{split}
	\end{align}
	A simple computation together with Young's inequality implies that 
	\[
	|\phi'(0_+)|^2\leq 2\int_0^\infty |\phi'\phi''|\,dy \leq 
	\left(\e \|\phi^\prime\|_{L^2(0,\infty)}^2+\frac{1}{\e}\|\phi^{\prime\prime}\|_{L^2(0,\infty)}^2 \right)
	\]
	for any $\e>0$. A similar estimate holds for $|\phi'(0_-)|^2$. Therefore
	\begin{align*}
		\left|\phi^\prime(0_+)-\phi^\prime(0_-)\right|^2 & \leq 2\left(|\phi^\prime(0_+)|^2+|\phi^\prime(0_-)|^2\right) \\
		& \leq 2\left(\e \|\phi^\prime\|_{L^2(\R)}^2+\frac{1}{\e}\|\phi^{\prime\prime}\|_{L^2((-\infty,0))}^2+\frac{1}{\e}\|\phi^{\prime\prime}\|_{L^2((0,\infty))}^2\right).
	\end{align*}
	Inserting the latter into \eqref{eq:ungleichung} yields 
	\begin{align*}
		\| L_\lambda^{k}\phi\|_{L^2(\R)}^2 \geq & \left(1-\frac{4(1-\lambda)}{\alpha \e}\right) (\|\phi^{\prime\prime}\|_{L^2(-\infty,0)}^2+\|\phi^{\prime\prime}\|_{L^2(0,\infty)}^2) \\
		& +2(1-\lambda)(k^2-\frac{2\e}{\alpha})\|\phi^{\prime}\|_{L^2(\R)}^2+k^4(1-\lambda)^2\|\phi\|_{L^2(\R)}^2.
	\end{align*}
	The choice $\e = \frac{8}{\alpha}(1-\lambda)=2\alpha k_*^2 + \frac{8(\lambda_*-\lambda)}{\alpha}$ implies the claim.
\end{proof}

Collecting Lemma \ref{prop:A1} and Lemma \ref{prop:A2}, we infer that there exists an open interval $I_{\lambda_*}\subset \R_+$ containing $\lambda_*$ such that conditions $(L0)$--$(L3)$ are satisfied, which concludes the proof of Corollary \ref{cor:p1}. The formulas for $\dot \lambda(0)$ and $\ddot \lambda(0)$ follow directly from Proposition \ref{prop:bf}.

\bigskip

\subsection{(P2) $V$ a $\delta$-potential on a step background}
\label{Ss:P2}

 Here we take $V_0(y)=\textbf{1}_{|y|\geq b}$, $V_1(y)=\beta  \textbf{1}_{|y|<b}+ \alpha \delta_0(y)$ and $V(\lambda,y)=\lambda \textbf{1}_{|y|\geq b}+ \beta  \textbf{1}_{|y|<b}+ \alpha\delta_0(y)$ with $\alpha>0$, $\beta,\lambda<1$. The subsequent two results verify that the family of linear operators $L_\lambda^k$ satisfies also the assumptions $(L2)$ and $(L3)$ in Theorem~\ref{main1}. They are the counterparts to Lemma~\ref{prop:A1} and Lemma~\ref{prop:A2}.

\begin{lem}\label{prop:A1_caseP2}
Let us fix a wavenumber $k_*\in \N$ and let $\lambda_*<1$. We determine $\alpha>0$ such that
\begin{equation}\label{eq:bifurcation_condition_caseP2}
\frac{k_*\alpha}{2\sqrt{1-\beta}} = \frac{\sqrt{1-\beta}\sinh(k_*\sqrt{1-\beta}b)+\sqrt{1-\lambda_*}\cosh(k_*\sqrt{1-\beta}b)}{\sqrt{1-\beta}\cosh(k_*\sqrt{1-\beta}b)+\sqrt{1-\lambda_*}\sinh(k_*\sqrt{1-\beta}b)}.
\end{equation}
Then there exists an open interval $I_{\lambda_*}\subset \R$ containing $\lambda_*$ such that
	\[
	\dim \ker L_{\lambda_*}^{k_*}=1,
	\]
	and $\ker L_{\lambda}^k=\{0\}$ for any $(k,\lambda)\in \N\times I_{\lambda_*}$ with $(k,\lambda)\neq (k_*,\lambda_*)$.
\end{lem}

\begin{proof} As before we are considering the zero-eigenvalue of $L_{\lambda}^{k}$. Hence we have $\mu=0$ and $\tilde\lambda=\lambda$, $\tilde\beta=\beta$. Then \eqref{eq:bifurcation_condition_caseP2} amounts to $L_{\lambda_*}^{k_*}$ having a simple zero eigenvalue, cf. Lemma~\ref{eigenvalue_condition}. It remains to show that for no other value of $\lambda\in I_{\lambda_*}$ and $k\in \N$ there is a zero eigenvalue of $L_\lambda^k$. First note that for $\lambda$ in a bounded interval of $(-\infty,1)$ there are only finitely many values of $k\in\{1,\ldots,K\}$ which potentially also fulfill \eqref{eq:bifurcation_condition_caseP2} since the right-hand side is bounded in $k$ and the left-hand side tends to infinity as $k\to \infty$. Now we observe (by a standard calculation) that for fixed $\lambda=\lambda_*$, the right-hand side of \eqref{eq:bifurcation_condition_caseP2} divided by $k$ is monotone decreasing in $k$. Hence, for given $\lambda_*$ no other value of $k\in\{1,\ldots,K\}$ than $k_*$ fulfills \eqref{eq:bifurcation_condition_caseP2}. Finally, since $k\in\{1,\ldots,K\}$ needs to be integer valued, we can find a sufficiently small open interval $I_{\lambda_*}\subset (-\infty,1)$ containing $\lambda_*$ such that \eqref{eq:bifurcation_condition_caseP2} is fulfilled for $(\lambda,k)\in I_{\lambda_*}\times\N$ if and only if $(\lambda,k)=(\lambda_*,k_*)$. 
\end{proof}

\begin{lem}\label{prop:A2_caseP2}
	There exists an open interval $I_{\lambda_*}\subset \R$ containing $\lambda_*$ such that the following holds for all sufficiently large $k\in\N$ and all $\lambda\in I_{\lambda_*}$:
	if $L_\lambda^k \phi =f$ for some $f\in L^2(\R)$, then
	\begin{equation}\label{eq:Est_caseP2}
		\|\phi\|_{L^2(\R)}\lesssim \frac{1}{k^2}\|f\|_{L^2(\R)}\qquad\mbox{and}\qquad	\|\phi^\prime \|_{L^2(\R)}\lesssim \frac{1}{k}\|f\|_{L^2(\R)}.
	\end{equation}
	In particular, there exists a constant $c=c(k_*,|I_{\lambda_*}|)$, depending on $k_*$ and the size of the interval $I_{\lambda_*}$, such that 
	\begin{equation*}\label{eq:gap_P2}
		(-ck^2,ck^2)\subset \rho (L_\lambda^k) \quad \mbox{ for every} \quad k \mbox{ sufficiently large}, \lambda\in I_{\lambda_*}.
	\end{equation*} 
\end{lem}

\begin{proof} The proof consists of two parts. First we determine an interval $(-ck^2,ck^2)\subset \rho(L_\lambda^k)$ for all $\lambda\in I_{\lambda_*}$ and all sufficiently large $k$. This implies the first part of the estimate in \eqref{eq:Est_caseP2}. In the second part of the proof we will show the remaining part of \eqref{eq:Est_caseP2}.

\emph{Part 1:} Recall that $\sigma_{ess}(L_\lambda^k)=[k^2(1-\lambda),\infty)$, which is consistent with the desired result provided we choose $I_{\lambda_*}$ in such a way that it has a positive distance from $1$. Subject to this observation we take the bounded interval $I_{\lambda_*}$ from Lemma~\ref{prop:A1_caseP2} and diminish it in the following if necessary. Notice that \eqref{this_makes_an_eigenvalue} describes all eigenvalues of $L_{\lambda}^k$ of the form $\tilde\mu=k^2\mu$, where $\mu$ is so small that $\sup_{\lambda\in I_{\lambda_*}}\{\tilde\lambda=\lambda+\mu\}<1$ and $\tilde\beta=\beta+\mu<1$. Now observe that uniformly for $\lambda\in I_{\lambda_*}$ and $\mu\in [-\mu_0,\mu_0]$ for small $\mu_0>0$ the left-hand side of \eqref{this_makes_an_eigenvalue} tends to $\infty$ as $k\to \infty$ whereas the right-hand side stays bounded in $k$. Therefore the set $[-\mu_0 k^2,\mu_0 k^2]$ belongs to the resolvent of $L_\lambda^k$ for all $\lambda\in I_{\lambda_*}$ and all sufficiently large $k$.

\emph{Part 2:} We need to distinguish the operator $L_{\lambda}^k=-\frac{d^2}{dy^2}+k^2(1-\lambda-\alpha\delta_0)$ of case (P1) from its counterpart in (P2). Within this part of the proof let us denote it by $L_{\lambda,\beta}^k=-\frac{d^2}{dy^2}+k^2(1-\lambda \textbf{1}_{|y|\geq b}- \beta  \textbf{1}_{|y|<b}- \alpha\delta_0)$. Using Part 1 we find
$$
\|(L_{\lambda,\beta}^k)^{-1}\|_{L^2\to L^2} \leq \frac{1}{\dist(0,\sigma(L_{\lambda,\beta}^k)} \lesssim \frac{1}{k^2}.
$$
Therefore, with $f\in L^2(\R)$ and $\phi$ as in the hypothesis of the lemma, we get $\|\phi\|_{L^2(\R)} \lesssim \frac{1}{k^2} \|f\|_{L^2(\R)}$. The estimate for $\|\phi'\|_{L^2(\R)}$ is obtained as follows. We have 
$$
L_{\lambda,\beta}^k \phi = L_\lambda^k \phi + k^2(-\beta+\lambda)\textbf{1}_{|y|< b}\phi = f
$$
from which we deduce by using $\|(L_\lambda^k)^{-1}\|_{L^2\to H^1}\lesssim \frac{1}{k}$ for $k\gg 1$ from Lemma~\ref{prop:A2}
$$
\|\phi'\|_{L^2(\R)} \leq \|(L_\lambda^k)^{-1}\|_{L^2\to H^1} \|f-k^2(-\beta+\lambda)\textbf{1}_{|y|< b}\phi\|_{L^2(\R)} \lesssim \frac{1}{k}(\|f\|_{L^2(\R)}+k^2\|\phi\|_{L^2(\R)}) \lesssim \frac{1}{k} \|f\|_{L^2(\R)} 
$$
where in the last step we have used the result from Part 1. The finishes the proof of the lemma.
\end{proof}

Due to Lemma~\ref{prop:A1_caseP2} and Lemma~\ref{prop:A2_caseP2} conditions $(L0)$--$(L3)$ are satisfied. This concludes the proof of Corollary~\ref{cor:p2}.

\bigskip
 
 \section{Existence of  traveling waves when $\Gamma$ is a delta potential}
 \label{S:Gamma_d}

Subject of this section is the proof of Theorem \ref{main2} when $\Gamma=\gamma \delta_0$ is given by a multiple of a delta potential and
 \[
 V(\lambda,y)=\lambda V_0(y)+\underbrace{W(y)+\alpha \delta_0(y)}_{=V_1(y)},
 \]
 where $V_0,W\in L^\infty(\R)$ are even.
The equation for traveling wave solutions \eqref{eq:Q_trav} is then given by
 \begin{equation}\label{eq:traveling_wave}
	 -\Phi_{yy}-(1-\lambda V_0(y)-W(y)-\alpha\delta_0(y))\Phi_{xx}+\gamma\delta_0(y)\left(\Phi^3\right)_{xx}=0
 \end{equation}
 and can be written as a linear partial differential equation on $\T\times \R\setminus\{0\}$ equipped with a nonlinear boundary condition on $\T$:
  \begin{align}
  -\Phi_{yy}-(1-\lambda V_0-W(y))\Phi_{xx}&=0, &&\quad (x,y)\in\T \times \R\setminus\{0\}, \label{eq:gleichung}\\
  \Phi_y(x,0_+)-\Phi_y(x,0 _-)&=\partial_x^2\left(\alpha \Phi+ \gamma\Phi^3\right)(x,0), &&\quad x\in\T. \label{eq:rb}
  \end{align}
  
  In what follows let us assume that $\Phi$ is even with respect to $y$. We seek for solutions $\Phi$ of the form
  \begin{equation}
  \Phi(x,y)=\sum_{k\in \N}a_k\phi_k(y;\lambda)\sin(kx), \label{ansatz}
  \end{equation}
  where $\phi_k(\cdot;\lambda)\in H^1(\R)\cap H^2(\R\setminus\{0\})$ is an evenly extended solution to the linear problem
  \begin{equation}\label{eq:linear}
   L_{0,k}^\lambda \phi_k(y;\lambda)=0 \quad \mbox{ on } (0,\infty) \qquad\mbox{ with } \quad \phi_k(0;\lambda)=1
  \end{equation}
  and
  $$
   L_{0,k}^\lambda := -\frac{d^2}{dy^2}+k^2(1-\lambda V_0(y)-W(y)).
  $$
  Thus ansatz \eqref{ansatz} already solves \eqref{eq:gleichung} and its remains to determine $a=(a_k)_{k\in \N}$ such that \eqref{eq:rb} is also satisfied. It will be convenient to parameterize the sequence $(a_k)$ over $\Z$ instead of $\N$ by setting  $a_k=-a_{-k}$. In this way $\Phi(x,y)=\frac{1}{2}\sum_{k\in \Z}a_k\phi_k(y;\lambda)\sin(kx)$. Here we have defined $\phi_{-k}(\cdot;\lambda):= \phi_k(\cdot;\lambda)$ for $k\in \N$. Then, we shall see that for $s\geq \frac{5}{2}$ the existence of a traveling wave solution $\Phi$ in the space
$$
X_s= H^s(\T;L^2(\R))\cap H^{s-1}(\T;H^1(\R))\cap C(\R;H^{s-\frac{1}{2}}(\T))\cap C^1(\R; H^{s-\frac{3}{2}}(\T))
$$
follows from the existence of a sequence $a\in h_\sharp^s(\R)$ satisfying the boundary condition
   \begin{equation}\label{eq:R}
   2a_k\phi_k^\prime(0_+;\lambda)=-k^2\left(\alpha a_k-\frac{\gamma}{4} (a*a*a)_k\right)\qquad\mbox{for all}\quad k\in \N.
   \end{equation}
   Recall that the $k$-th Fourier coefficient of $\Phi^3(x,0)$ is given by $-\frac{1}{4}(a*a*a)_k$ (cf. Lemma~\ref{A1}). Notics that for $s\geq \frac{5}{2}$ we have the embedding
$$ X_s \hookrightarrow X= H^2(\T;L^2(\R))\cap H^1(\T;H^1(\R))\cap C(\R;H^2(\T))\cap C^1(\R; H^1(\T)).
$$

   \medskip
   
   As in the previous section we aim to apply bifurcation theory with respect to the the parameter $\lambda$ to obtain the existence of nontrivial solutions $a\in h_\sharp^s(\R)$ of \eqref{eq:R} for $s\geq \frac{5}{2}$ by the Crandall--Rabinowitz theorem. Recall that this time we are working under the the assumptions $(\tilde L0)$--$(\tilde L3)$ from Theorem~\ref{main2}.

\begin{rem}\label{rem:existenz_decaying_mode} 
	\emph{
		Existence and properties of the decaying solutions $\phi_k(\cdot;\lambda)$ of \eqref{eq:linear}:
	\begin{itemize} 
		\item[(i)] Due to $(\tilde L1)$ the problem $L_{0,\lambda}^k \phi_k = \textbf{1}_{[-2,-1]}$ on $\R$ has a unique $H^2(\R)$ solution. Its restriction to $[0,\infty)$ satisfies \eqref{eq:linear}. The fact that $\phi_k(y;\lambda)\to 0$ exponentially as $y\to\infty$ can be seen as follows: 
		Since $L_{0,k}^\lambda$ is a self-adjoint operator with $0\in \rho(L_{0,k}^\lambda)$  and the resolvent set is open in $\C$ there exists $c_{k,\lambda}>0$ such that $(-c_{k,\lambda},c_{k,\lambda})\subset \rho(L_{0,k}^\lambda)$.	Set $\psi_k(y;\lambda):=e^{\delta_ky}\phi_k(y;\lambda)$. Then 
	\begin{equation}\label{eq:psi}
	L_{0,k}^\lambda \psi_k(y;\lambda) + B_k \psi_k(y;\lambda) = e^{\delta_ky}\textbf{1}_{[-2,-1]},
	\end{equation}
	where $B_k\psi:=2\delta_k \frac{d}{dy} \psi +\delta_k^2 \psi$. One can show that $B_k$ is $L_{0,k}^\lambda$--bounded in the sense that there exist $a_k,b_k>0$ such that
	\[
	\|B_k\psi\|_{L^2(\R)}^2\leq a_k \|\psi\|_{L^2(\R)}^2 + b_k\|L_{0,k}^\lambda \psi\|_{L^2(\R)}^2\qquad \mbox{for all}\quad \psi \in H^2(\R).
	\]
	In fact, if $b_k >0$ is fixed, then $a_k := \frac{16 \delta_k^4}{b_k}+8\delta_k^2 k^2\|1-\lambda V_0-W\|_\infty+2\delta_k^4$. For fixed $b_k\in (0,1)$ let us choose $\delta_k>0$ so small that
\[
a_k^2 +b_k^2 c_{k,\lambda}^2< c_{k,\lambda}^2.
\]	
Then $(-\tilde c_{k,\lambda}, \tilde c_{k,\lambda}) +i\R \subset \rho(L_{0,k}^\lambda +B_k)$, where $\tilde c_{k,\lambda}=c_k - \sqrt{a_k^2 +b_k^2 c_{k,\lambda}^2}$, cf. \cite[Theorem 2.1 (ii)]{cuenin_tretter}. In particular, $0\in\rho(L_{0,k}^\lambda +B)$ so that there exists a unique solution $\psi_k\in H^2(\R)$ of \eqref{eq:psi}. The boundedness of $\psi_k$ then implies that $|\phi_k(y;\lambda)|\lesssim e^{-\delta_k y}$ decays exponentially on the half-line $[0,\infty)$. This result is also known as ``exponential dichotomy''. Assumption $(\tilde L3)$  may be interpreted as some kind of generalized uniform exponential dichotomy with respect to $k\in \N$ and $\lambda\in I_{\lambda_*}$. 
		\item[(ii)] In the specific examples (P1) and (P2) which we consider at the end of this section, the family of ansatz-functions $(\phi_k(\cdot;\lambda))_{k\in \N}$ satisfies a true uniform exponential dichotomy with respect to $k\in \N$ and $\lambda \in I_{\lambda_*}$; that is, there exists $C,\delta>0$ independent of $k\in \N$ and $\lambda \in I_{\lambda_*}$ such that $|\phi_k(y;\lambda)|\leq Ce^{-\delta y}$ for all $y\geq 0$. This leads to an exponential decay in $y$-direction of the traveling solution $\Phi$ of \eqref{eq:Q_trav} and in particular it implies $(\tilde L3)$.
		\item[(iii)] Notice also that $\phi_k(0;\lambda)\neq0$, since otherwise (by an odd reflection around zero) we would obtain an eigenfunction of $L_{0,k}^\lambda$ for the eigenvalue $0$. This is excluded by assumption $(\tilde L1)$. Likewise we see that $\phi_k'(0;\lambda)\neq 0$ (using an even reflection around zero).
	\end{itemize}
}
\end{rem}

\begin{rem}\label{rem:pos} \emph{
	If $V_0,W$ are bounded, even functions and there exists $\bar v>0$ such that
	 \[
	1-\lambda V_0(y)-W(y) \geq \bar v\qquad \mbox{ for all } \lambda \in I_{\lambda_*}, y\in\R,
		\]
		 then assumption $(\tilde L0), (\tilde L1),$ and $(\tilde L3)$ are satisfied. Clearly, if $1-\lambda V_0-W\geq \bar v$, then $L_{0,k}^\lambda$ is a self-adjoint operator with $\sigma(L_{0,k}^\lambda)\subset [k^2\bar v,\infty)$; thus $0\in \rho(L_{0,k}^\lambda)$ and $(\tilde L1)$ is satisfied. As explained in Remark \ref{rem:existenz_decaying_mode} (i), condition $(\tilde L1)$ implies the existence of a solution $\phi_k(\cdot, \lambda) \in H^2(0,\infty)$ with
	\begin{equation}\label{eq:b}
		-\phi_k^{\prime\prime}+k^2(1-\lambda V_0(y)-W(y))\phi_k=0 \qquad \mbox{on}\quad (0,\infty)
	\end{equation}
and $\phi_k(0;\lambda)=1$.
Multiplying \eqref{eq:b} with $\phi_k$ and integrating over the half line $(0,\infty)$, we obtain that
	\begin{equation}\label{eq:phi1}
		-\phi_k^\prime(0_+;\lambda)= \int_{0}^\infty |\phi_k^\prime|^2\,dy + k^2 \int_0^\infty (1-\lambda V_0-W)\phi_k^2 \,dy.
	\end{equation}
	On the other hand, multiplying \eqref{eq:b} with $\phi_k^\prime$ and integrating over $(0,\infty)$ yields
	\begin{align*}
		(\phi_k^\prime)^2(0_+;\lambda)&= -k^2 \int_{0}^\infty (1-\lambda V_0-W)(\phi_k^2)'\,dy \\
		&=-2k \int_0^\infty k \sqrt{1-\lambda V_0-W}\phi_k \phi_k^\prime \sqrt{1-\lambda V_0-W}\,dy\\
		&\leq k \|\sqrt{1-\lambda V_0-W}\|_\infty \left(\int_{0}^\infty |\phi_k^\prime|^2\,dy + k^2 \int_0^\infty (1-\lambda V_0-W)\phi_k^2 \,dy \right)\\
		&= -k \|\sqrt{1-\lambda V_0-W}\|_\infty\phi_k^\prime(0_+;\lambda),
	\end{align*}
	where we used relation \eqref{eq:phi1} in the last equality. We deduce that $\phi^\prime(0_+;\lambda)<0$ and
	\begin{equation}\label{eq:phi2}
		|\phi_k^\prime(0_+;\lambda)|\leq k\|\sqrt{1-\lambda V_0-W}\|_\infty.
	\end{equation}
	Estimating the $L^2$-norm of $\phi_k(\cdot;\lambda)$, we obtain that
	\begin{align*}
		\|\phi_k(\cdot;\lambda)\|_2^2 =\frac{1}{\bar v}\|\sqrt{\bar v} \phi_k(\cdot;\lambda)\|_2^2\leq \frac{1}{\bar v}\int_0^\infty (1-\lambda V_0-W)\phi_k^2 \,dy \leq \frac{1}{\bar v k^2} |\phi_k^\prime(0_+;\lambda)|\leq \frac{1}{\bar v k} \|\sqrt{1-\lambda V_0-W}\|_\infty,
	\end{align*} 
	where we used \eqref{eq:phi1} and \eqref{eq:phi2}. In particular, we find that $\|\phi_k(\cdot;\lambda )\|_2 \lesssim 1$ as claimed in $(\tilde L3)$.
}
	\end{rem}

%
%

\medskip

For $s\geq 0$ denote the linearization of \eqref{eq:R} around $a=0$ by
   \[
	   A_\lambda:h_\sharp^{s}(\R)\subset h_\sharp^{s-2}(\R)\to h_\sharp^{s-2}(\R),\qquad (A_\lambda a)_k:=A_\lambda^ka_k \mbox{ for } k \in\Z,
   \]
   where 
   \[
   A_\lambda^k:=
   2\phi^\prime_k(0_+;\lambda)+k^2\alpha \mbox{ for } k\in \Z.
   \]
   Then \eqref{eq:R} can be written as 
      \begin{equation}\label{eq:A}
   	   A_\lambda a -n(a)=0,\qquad \mbox{where}\qquad n(a)_k= \frac{\gamma k^2}{4}(a*a*a)_k.
      \end{equation}
For $m\in \Z$ let us denote by $e^{m}\in l^2_\sharp(\R)$ the sequence, which satisfies $e^{m}_k=0$ for $k\neq \pm m$ and $e^{m}_{m}=-e^{m}_{-m}=\frac{1}{\sqrt{2}}$. 

\begin{lem} Assume $(\tilde L0)$--$(\tilde L3)$. Then 
\begin{equation} \label{phi_k_estimates}
\|\phi_k(\cdot;\lambda)\|_{L^\infty(\R)} \lesssim k^\frac{1}{2}, \quad \|\phi_k'(\cdot;\lambda)\|_{L^2(\R)} \lesssim k, \quad \|\phi_k'(\cdot;\lambda)\|_{L^\infty(\R)} \lesssim k^\frac{3}{2}
\end{equation}
uniformly for $\lambda \in I_{\lambda_*}$. In particular, $|\phi_k'(0;\lambda)| \lesssim |k|^\frac{3}{2}$ and consequently $A_\lambda^k = \alpha k^2 + O(k^\frac{3}{2})$ as $k\to \pm\infty$. 
\label{asymptotik}
\end{lem}

\begin{proof} By a result of Komornik, cf.~\cite{komornik}, the estimate
$$
\|u\|_\infty \leq C\|u\|_{L^2}
$$
holds true for every solution $u$ of $-u'' +q(y) u=0$ on $(a,b)$ with the constant $C=\max\Bigl\{6\sqrt{\|q\|_{L^1(a,b)}}, \frac{12}{\sqrt{b-a}}\Bigr\}$. We apply this result to the solutions $\phi_k(\cdot;\lambda)$ of \eqref{eq:linear} with $q=k^2(1-\lambda V_0-W)$, $a\geq 0$ and $b=a+c$ with $c:= 2 (\sqrt{\|1-\lambda V_0-W\|_\infty} k)^{-1}$. Then 
$$
6\sqrt{\|q\|_{L^1(a,b)}} \leq 6k\sqrt{\|1-\lambda V_0-W\|_\infty}\sqrt{c} = 6\sqrt{2}\sqrt{k} \sqrt[4]{\|1-\lambda V_0-W\|_\infty}=\frac{12}{\sqrt{b-a}}
$$
and thus for a constant $\tilde C$ only depending on $\|1-\lambda V_0-W\|_\infty$ we have $\|\phi_k(\cdot;\lambda)\|_{L^\infty(a,b)} \leq k^\frac{1}{2} \tilde C \|\phi_k(\cdot;\lambda)\|_{L^2(a,b)}\lesssim k^\frac{1}{2}$ by $(\tilde L3)$. Since $a\geq 0$ was arbitrary we obtain the first part of \eqref{phi_k_estimates}.

\medskip

Multiplying \eqref{eq:linear} with $u$, $u'$ and integrating from $a\geq 0$ to $\infty$ we get 
\begin{align}
\int_a^\infty k^2 (1-\lambda V_0(y)-W(y))\phi_k(y;\lambda)^2+\phi_k'(y;\lambda)^2\,dy & = -\phi_k(a;\lambda)\phi_k'(a;\lambda), \label{mult1}\\
\int_a^\infty 2k^2 (1-\lambda V_0(y)-W(y))\phi_k(y;\lambda)\phi_k'(y;\lambda)\,dy & = -\phi_k'(a;\lambda)^2, \label{mult2}
\end{align}
respectively. Using $(\tilde L3)$ and applying the Cauchy-Schwarz inequality to \eqref{mult2} we find 
\begin{equation} \label{mult3}
\|\phi_k'(\cdot;\lambda)\|_{L^\infty}^2 \lesssim k^2 \|\phi_k'(\cdot;\lambda)\|_{L^2(0,\infty)}
\end{equation}
and from \eqref{mult1}, \eqref{mult3} we get 
\begin{align*}
\|\phi_k'(\cdot;\lambda)\|_{L^2(0,\infty)}^2 & \lesssim k^2 + \|\phi_k(\cdot;\lambda)\|_{L^\infty(0,\infty)} \|\phi_k'(\cdot;\lambda)\|_{L^\infty(0,\infty)}\\
& \lesssim k^2 + \|\phi_k(\cdot;\lambda)\|_{L^\infty(0,\infty)} k \|\phi_k'(\cdot;\lambda)\|_{L^2(0,\infty)}^\frac{1}{2}.
\end{align*}
The $L^\infty$--estimate from the first part of the lemma leads to 
$$
\|\phi_k'(\cdot;\lambda)\|_{L^2(0,\infty)}^2  \lesssim k^2 + k^\frac{3}{2}\|\phi_k'(\cdot;\lambda)\|_{L^2(0,\infty)}^\frac{1}{2} \leq k^2 +C_\epsilon k^2 + \epsilon \|\phi_k'(\cdot;\lambda)\|_{L^2(0,\infty)}^2,
$$
where we have used Young's inequality with exponents $4/3$ and $4$. This implies the second inequality in \eqref{phi_k_estimates}.  Inserting this into \eqref{mult3} we obtain the third inequality in \eqref{phi_k_estimates}. 
\end{proof} 

\begin{lem} Assume $(\tilde L0)$--$(\tilde L3)$ and let $s\geq 0$. Then the operator $A_\lambda: h_\sharp^{s}(\R)\subset h_\sharp^{s-2}(\R)\to h_\sharp^{s-2}(\R)$ is self-adjoint. Its spectrum is discrete and consist of the values $(A_\lambda^k)_{k\in \N}$. Moreover $\ker A_{\lambda_*}=\spa\{e^{k_*}\}$ and $\ker A_{\lambda}=\{0\}$ for $\lambda\in I_{\lambda_*}\setminus\{\lambda_*\}$.
\label{self_adjoint}  
\end{lem}

\begin{proof} Due to Lemma~\ref{asymptotik} and since  $\phi_{-k} = \phi_k$ for all $k\in \Z$ one can verify that $A_\lambda: h_\sharp^{s}(\R)\subset h_\sharp^{s-2}(\R)\to h_\sharp^{s-2}(\R)$ acting like an infinite dimensional diagonal matrix is self-adjoint. Using the characterization of the spectrum via Weyl-sequences one sees that $A_\lambda$ has the spectrum $\sigma(A_\lambda)=\clos\{A_\lambda^k: k\in \N\}$. Due to Lemma~\ref{asymptotik} the set $\{A_\lambda^k: k\in \N\}$ is discrete and hence $\sigma(A_\lambda^k)$ consists of the set of eigenvalues $\{A_\lambda^k: k\in \N\}$. Finally, let us determine the kernel of $A_\lambda$. On the one hand, $a\in \ker A_\lambda$ if and only if there exists $k\in \N$ such that $A_\lambda^k=0$, and in this case $a=e^k$ (here we use that $A_\lambda^k=A_\lambda^{-k}$). On the other hand, using the characterization of the domain of $L_k^\lambda$ from Section~\ref{Ss:P1} we know that $L_k^\lambda\phi =0$ if and only if $\phi(x)=\phi_k(|x|;\lambda)$ and $A_\lambda^k=0$. Thus, bringing both facts together and using assumption $(\tilde L2)$ we obtain the final claim of the lemma.
\end{proof}
Similarly as in \eqref{def_projectedop} we define the operator 
\[
	\tilde A_\lambda:= A_\lambda + P^*,
\]
where $P^*a=a_{k_*}e^{k_*}$. 

\begin{lem}\label{lem:FA} Assume $(\tilde L0)$--$(\tilde L3)$ and $s\geq 0$. Then we have that $0\in \rho(\tilde A_\lambda)$ for all $\lambda \in I_{\lambda_*}$ and hence $\tilde A_\lambda^{-1}: h_\sharp^{s-2} (\R)\to h_\sharp^{s}(\R)$ is a bounded linear operator. Moreover, if $f\in h^s_\sharp(\R)$, $s\geq \frac{5}{2}$ is given and $a\in h^s_\sharp(\R)$ solves 
\begin{equation} \label{eq:inv_delta} 
a = \tilde A_\lambda^{-1}(-\gamma M(f)+P^*a) \qquad \mbox{where}\qquad M(f)_k := k^2 f_k
\end{equation}
then $\Phi(x,y) := \sum_{k\in \N} a_k \phi_k(x;\lambda)\sin(kx)$ satisfies $\Phi\in X_s$ and solves 
\begin{equation*} \label{eq:direkt_delta}
L_\lambda \Phi + \gamma\delta_0(y) F_{xx}=0
\end{equation*}
for $F(x,y) = \sum_{k\in \N} f_k \phi_k(y;\lambda)\sin(kx)$ in the weak sense, i.e.,
	\[
	\int_{\T}\int_{\R} \Phi_y\Psi_y - \left(1-\lambda V_0(y)-W(y)\right)\Phi_{xx}\Psi\,dy\,dx + \int_\T \bigl(\alpha\Phi_{xx}(x,0)+ \gamma F_{xx}(x,0)\bigr)\Psi(x,0) \,dx=0
	\]
	for any $\Psi \in H^1(\T; H^1(\R))$. 
\end{lem}

\begin{proof} Lemma~\ref{self_adjoint} says that $\ker A_{\lambda_*}=\spa \{e^{k_*}\}$ and $\ker A_\lambda=\{0\}$ for any $\lambda \in I_{\lambda_*}\setminus\{\lambda_*\}$. We need to show that $0\in \rho (\tilde A_\lambda)$ for any $\lambda\in I_{\lambda_*}$. Let $b\in h^{s-2}_\sharp(\R)$ be arbitrary, then $\tilde A_\lambda a = b$ if and only if
   		\begin{equation*}
   		\left\{\begin{array}{lcl}  A_\lambda^k a_k=b_k,\quad && \mbox{if}\quad k\neq k_*,\\ A_\lambda^ka_k+a_k=b_k,\quad && \mbox{if}\quad k= k_*,\end{array} \right.
   		\end{equation*}
   		which is equivalent to
   		\begin{equation} \label{def_ak}
	   		a_k= \frac{1}{A_\lambda^k}b_k\quad \mbox{if}\quad k\neq k_*\qquad\mbox{and}\qquad 	a_{k_*}= \frac{1}{A_\lambda^{k_* }+1}b_{k_*}.
   		\end{equation}
   		Due to Lemma~\ref{asymptotik} we obtain that for any $b\in h_\sharp^{s-2}(\R)$ the sequence $a$ defined by \eqref{def_ak} belongs to $h_\sharp^{s}(\R)$  and solve $\tilde A_\lambda a=b$; whence $0\in \rho(\tilde A_\lambda)$.

   		Now suppose that $f\in h_\sharp^s(\R)$ with $s\geq \frac{5}{2}$ and that $a\in h_\sharp^s(\R)$ solves \eqref{eq:inv_delta}. The regularity of $\Phi$ follows from Lemma~\ref{lem:regularity}. Moreover, $(A_\lambda a)_k=-\gamma k^2 f_k$ and hence 
\begin{equation} \label{formel_phikstrich}
2\phi_k'(0_+;\lambda)a_k +\alpha k^2 a_k= -\gamma k^2 f_k.
\end{equation}
Using that 
$$
L_{0,\lambda}^k \phi_k = 0 \mbox{ on } \R\setminus\{0\} \mbox{ for all } k\in \N,
$$
we deduce by testing with $a_k\psi_k\in H^1(\R)$ and summing for $1\leq k\leq K_0$ that 
$$
0=\sum_{k=1}^{K_0}\int_\R a_k\phi_k' \psi_k' + k^2(1-\lambda V_0(y)-W(y))\phi_k\psi_k \,dy+ a_k(\phi_k'(0_+;\lambda)-\phi_k'(0_-;\lambda))\psi_k(0).
$$
Since $\phi_k(\cdot;\lambda)$ is even with respect to $y$ we obtain by \eqref{formel_phikstrich}
$$
0=\sum_{k=1}^{K_0}\int_\R a_k\phi_k' \psi_k' + k^2(1-\lambda V_0(y)-W(y))\phi_k\psi_k \,dy -\alpha k^2 a_k\underbrace{\phi_k(0;\lambda)}_{=1}\psi_k(0)-\gamma k^2 f_k \psi_k(0). 
$$
Taking the limit $K_0\to\infty$ in the previous equation will lead to 
$$
0=\int_{\T}\int_{\R} \Phi_y\Psi_y - \left(1-\lambda V_0(y)-W(y)\right)\Phi_{xx}\Psi+ \int_{\T} (\alpha \Phi_{xx}(x,0)+\gamma F_{xx}(x,0)) \Psi(x,0)\,dx
$$
for any $\Psi \in H^1(\T;H^1(\R))$ due to the following estimates:
\begin{align*}
\int_{\T}\int_{\R} |\Phi_y\Psi_y| \,dy\,dx & \leq \sum_{k\in \N} \|a_k \phi_k'\|_{L^2(\R)} \|\psi_k'\|_{L^2(\R)} \leq \|\Phi\|_{L^2(\T,H^1(\R))} \|\Psi\|_{L^2(\T,H^1(\R))} \\
& \leq C \|a\|_{h^1(\R)} \|\Psi\|_{L^2(\T,H^1(\R))}, \\
\int_{\T}\int_{\R} |\bigl(1-\lambda V_0(y)-W(y)\bigr) \Phi_{xx} \Psi| \,dy\,dx & \leq \sum_{k\in \N} \|1-\lambda V_0-W\|_{L^\infty(\R)} k^2\|\phi_k\|_{L^2(\R)} \|\psi_k\|_{L^2(\R)} \\ &  \leq \|1-\lambda V_0-W\|_{L^\infty(\R)}  \|\Phi\|_{H^2(\T,L^2(\R))} \|\Psi\|_{L^2(\T,L^2(\R))} \\
& \leq C \|a\|_{h^2(\R)}\|\Psi\|_{L^2(\T,L^2(\R))} \\
\int_{\T} |\Phi_{xx}(x,0) \Psi(x,0)|\,dx & \leq \sup_y \|\Phi_{xx}(\cdot,y)\|_{L^2(\T)} \|\Psi(\cdot,y)\|_{L^2(\T)} \\
& \leq \|\Phi\|_{C(\R,H^2(\T))}\|\Psi\|_{C(\R,L^2(\T))} \\
& \leq C \|a\|_{h^\frac{5}{2}(\R)} \|\Psi\|_{C(\R,L^2(\T))}, \\
\int_{\T} |F_{xx}(x,0) \Psi(x,0)|\,dx & \leq \|f\|_{h^\frac{5}{2}(\R)} \|\Psi\|_{C(\R,L^2(\T))}
\end{align*}
together with the continuous embeddings $H^1(\T;H^1(\R)) \subset L^2(\T;H^1(\R))\cap C(\R;L^2(\T))$ and $h_\sharp^s(\R)\subset h_\sharp^{5/2}(\R)\subset h^2_\sharp(\R)\subset h^1_\sharp(\R)$ since $s\geq \frac{5}{2}$.  
\end{proof}
   
In the same spirit as in Section \ref{S:Gamma_b}, let us reformulate our problem \eqref{eq:A} in a way suitable for applying the Crandall--Rabinowitz theorem. Using the above lemma, equation \eqref{eq:A} is equivalent to
   \[
   G(a,\lambda)=0,
   \]
   where the function $G:h_\sharp^s(\R)\times I_{\lambda_*}\to h_\sharp^s(\R)$, $s\geq \frac{5}{2}$, is defined by
   \begin{equation} \label{def_M}
   G(a,\lambda):=a+\tilde A_\lambda^{-1}\left(-\frac{\gamma}{4} M(a*a*a)-P^*a\right) \qquad \mbox{and}\qquad M(f)_k := k^2 f_k \quad \mbox{for} \quad f\in h^s_\sharp(\R).
   \end{equation}
   
   \begin{rem} \label{algebra_and_mapping}
   	\emph{
   	Notice that $h_\sharp^s(\R)$, $s\geq 1$ is a Banach algebra, cf. Lemma~\ref{lem:algebra}. Thus, for $a\in h_\sharp^s(\R)$ the nonlinearity $a*a*a$ stays in $h_\sharp^s(\R)$ and $M(a*a*a)\in h_\sharp^{s-2}(\R)$. Hence, in order to control the nonlinearity in $G(a,\lambda)$, it is necessary that $\tilde A_\lambda^{-1}$ is a bounded operator from $h_\sharp^{s-2}(\R)$ to $h_\sharp^{s}(\R)$. Otherwise, assume that we would only have that $\tilde  A_{\lambda}^{-1}$ is bounded from $h_\sharp^{s-2}(\R)$ to $h_\sharp^{s'}(\R)$ where $s'<s$,
   	then the mapping $G$ is merely bounded from $h_\sharp^s(\R)\times I_{\lambda}\to h_\sharp^{s'}(\R)$. In this case, the Fr\'echet derivative has the property that $D_aG(0,\lambda):h_\sharp^s(\R)\to h_\sharp^s(\R)$ (cf. Lemma~\ref{lem:prop_G}(ii) below) but is no longer a Fredholm operator from $h_\sharp^s(\R)\to h_\sharp^{s'}(\R)$ since the co-dimension of its image is infinite. The Fredholm property at $\lambda=\lambda_*$, however, is important for applying the Crandall--Rabinowitz theorem for bifurcation.
   }
   	\end{rem}

The following lemma provides the necessary preparations to apply bifurcation theory to $G(a,\lambda)=0$.

\begin{lem} \label{lem:prop_G} Let $s\geq \frac{5}{2}$.
The map $G:h_\sharp^s(\R)\times I_{\lambda_*} \to h_\sharp^s(\R)$ is a $C^\infty$-map. Moreover the following holds:
\begin{itemize}
\item[(i)] The function $\phi_k$ is continuously differentiable with respect to $\lambda$ and $\psi_k(y;\lambda) := \partial_\lambda\phi_k(y;\lambda)$ satisfies 
\begin{equation} \label{psi_k_def}
L^\lambda_{0,k} \psi_k = k^2 V_0(y) \phi_k\mbox{ on } (0,\infty), \quad \psi_k(0;\lambda)=0
\end{equation}
and $\psi_k'(0;\lambda)=k^2\int_0^\infty V_0 \phi_k^2\,dy=O(k^2)$.
\item [(ii)] The linearization of $G$ about $a=0$, given by
\[
D_a G(0,\lambda)= \Id - \tilde A_\lambda^{-1}P^*:h_\sharp^s(\R) \to h_\sharp^s(\R)
\]
is a Fredholm operator of index zero. Its kernel is trivial for $\lambda \in I_{\lambda_*}$, $\lambda\neq\lambda_*$ and it is given by $\spa\{e^{k_*}\}$ if $\lambda=\lambda_*$. 
\item[(iii)] The mixed second derivative of $G$ about $a=0$ is given by 
\begin{equation*} \label{partial_alambda}
D^2_{a\lambda} G(0,\lambda)= \tilde A_\lambda^{-1} B \tilde A_\lambda^{-1} P^*: h^s_\sharp(\R)\to h_\sharp^{s+2}(\R)\subset h_\sharp^s(\R),
\end{equation*}
where $B:h_\sharp^{s+2}(\R)\to h_\sharp^s(\R)$ is the pointwise multiplication with $2\psi_k'(0;\lambda)$. 
\end{itemize}
\end{lem}

\begin{proof} (i) We are only interested in $\psi_k(\cdot,\lambda)=\partial_\lambda\phi_k(\cdot,\lambda)$ on $[0,\infty)$. To find $\psi_k(\cdot,\lambda)$ we differentiate $L_{0,k}^\lambda \phi_k=0$ on $(0,\infty)$, $\phi_k(0;\lambda)=1$ with respect to $\lambda\in I_{\lambda_*}$ and obtain \eqref{psi_k_def}. If we define $Q\phi_k(\cdot,\lambda):\R\to\R$ as the odd extension around $y=0$ of $\phi_k(\cdot,\lambda): (0,\infty)\to\R$ then we see that $\psi_k(\cdot,\lambda)$ is given by $\psi_k(\cdot,\lambda) = (L^\lambda_{0,k})^{-1}(k^2 V_0Q\phi_k(\cdot,\lambda))|_{[0,\infty)}$. Testing the differential equation in \eqref{psi_k_def} with $\phi_k$ and noting that $L^\lambda_{0,k}\phi_k =0$ on $(0,\infty)$ we find 
\begin{align*}
\langle L^\lambda_{0,k}\psi_k,\phi_k\rangle_{L^2(0,\infty)} &= (-\psi_k'\phi_k+\psi_k\phi_k')\big|_0^\infty +\langle \psi_k,L^\lambda_{0,k}\phi_k\rangle_{L^2(0,\infty)} = \psi_k'(0_+;\lambda) \label{psi_k} = O(k^2),
\end{align*} 
as claimed, in view of $\langle L^\lambda_{0,k}\psi_k,\phi_k\rangle_{L^2(0,\infty)}=\int_0^\infty k^2 V_0 \phi_k^2\,dy=O(k^2)$ because $V_0\in L^\infty(\R)$ and $(\tilde L3)$.

(ii) The mapping properties of $G$ follow from Remark~\ref{algebra_and_mapping}. The differentiability of the cubic nonlinearity $a*a*a$ with respect to $a$ is also a straightforward property of the Banach algebra property of $h_\sharp^s(\R)$. The differentiability property of $G$ with respect to  $\lambda$ follows from differentiability of $\lambda \mapsto \phi'_k(0_+;\lambda)$ as given in (i). As in Lemma~\ref{lem:prop_F} the Fredholm property of $D_a(G(0,\lambda))$ is satisfied since it is a compact perturbation of the identity and the characterization of the kernel of can be seen in a similar way using Lemma~\ref{self_adjoint}.

(iii) Note that $\frac{d}{d\lambda} \tilde A_\lambda^k=\frac{d}{d\lambda} A_\lambda^k = 2\psi_k'(0;\lambda)$. Since $\psi_k'(0;\lambda)=O(k^2)$ by (i) we have the mapping property $B:h_\sharp^{s+2}(\R)\to h_\sharp^s(\R)$. 
 \end{proof}

We are now in a position to apply the Crandall--Rabinowitz theorem for $G(a,\lambda):h_\sharp^s(\R) \times I_{\lambda_*}\to X_s$ for $s\geq \frac{5}{2}$ in order to proof Theorem \ref{main2} provided that the \emph{transversality condition} in \eqref{eq:t2} is satisfied.

\medspace

\textbf{Proof of Theorem \ref{main2}.}
   	The existence result follows from the Crandall--Rabinowitz theorem applied to $G(a,\lambda)=0$. Successfully applied, it  provides an interval $I_{\lambda_*}\subset \R$ containing $\lambda_*$, and a smooth curve through $(0,\lambda_*)$ of the form
   	\[
   	\{(a(\e),\lambda(\e))\mid |\e|<\e_0\}\subset h^s(\R)\times I_{\lambda_*}
   	\] 
   	of nontrivial solutions of \eqref{eq:A} with $\lambda(0)=\lambda_*$ and $D_\e a(0)=e^{k_*}$. The curve $(a(\e),\lambda(\e))\subset h^s(\R)\times I_{\lambda_*}$ then translates via $\Phi(\e)(x,y)= \sum_{k\in\N} a_k(\e)\phi_k(y;\lambda)\sin(kx)$ and by Lemma~\ref{lem:FA} into the curve $\{(\Phi(\e),\lambda(\e))\mid |\e|<\e_0\}\subset X_s\times I_{\lambda_*}$ of nontrivial solutions of \eqref{eq:traveling_wave} with the stated property. The Crandall--Rabinowitz theorem requires that the linearization 
   	\[
   	D_aG(0,\lambda_*)=\Id -\tilde A_{\lambda_*}^{-1}P^*:h_\sharp^s(\R)\to h_\sharp^s(\R)
   	\]
   	is a Fredholm operator of index zero with $\dim \ker D_aG(0,\lambda_*)=1$ and the transversality condition
\begin{equation}\label{eq:trans}
D_{a\lambda}^2G(0,\lambda_*)e^{k_*} \not\in\range D_aG(0,\lambda_*)
\end{equation}
is satisfied. 
The Fredholm property is already shown in Lemma \ref{lem:prop_G} (ii) and the kernel of $D_aG(0,\lambda_*)$ is one dimensional and spanned by $e^{k_*}$, that is
\[
\ker D_aG(0,\omega_*) =\spa \{e^{k_*}\}.
\]
Concerning the transversality condition \eqref{eq:trans}, assume on the contrary that there exists $b\in h^s(\R)$ such that
\[
D_{a\lambda}^2G(0,\lambda_*)e^{k_*} =D_aG(0,\lambda_*)b.
\]
Then,
\[
\langle D_{a\lambda}^2G(0,\lambda_*)e^{k_*} ,e^{k_*}\rangle_{l^2(\R)}= \langle D_aG(0,\lambda_*)b,e^{k_*}\rangle_{l^2(\R)}.
\]
Using the formulas from Lemma~\ref{lem:prop_G} (ii) and (iii) and the fact that $\tilde A_{\lambda_*}^{-1}P^*=P^*$ together with the symmetry of $\tilde A_{\lambda_*}^{-1}$ we obtain that
\begin{align}\label{transverse}
	\begin{split}
2\psi_{k_*}'(0;\lambda_*)&=\langle B e^{k_*} ,e^{k_*}\rangle_{l^2}=\langle D_{a\lambda}^2G(0,\lambda_*)e^{k_*} ,e^{k_*}\rangle_{l^2(\R)}= \langle D_aG(0,\lambda_*)b,e^{k_*}\rangle_{l^2(\R)}\\
&=\langle b-P^*b, e^{k_*}\rangle_{l^2(\R)}=0. 
\end{split}
\end{align}
But due to Lemma \ref{lem:prop_G} (i)  this is a contradiction to \eqref{eq:t2}.

\medspace
        	     
   Similarly as in the previous section, we determine the bifurcation formulas. The Fr\'echet derivatives of $G$ with respect to $a$ are given by
   \begin{align*}
   D_aG(a,\lambda)e^{k_*}&=e^{k_*}-\tilde A_\lambda^{-1}\left(\frac{3}{4}\gamma M(a*a*e^{k_*})+ e^{k_*}\right)\\
   D_{aa}^2G(a,\lambda)[e^{k_*},e^{k_*}]&=-\frac{3}{2}\gamma \tilde A_\lambda^{-1}\left(M(a*e^{k_*}*e^{k_*})\right)\\
   D_{aaa}^3G(a,\lambda)[e^{k_*},e^{k_*},e^{k_*}]&=-\frac{3}{2}\gamma \tilde A_\lambda^{-1}\left(M(e^{k_*}*e^{k_*}*e^{k_*})\right)
   \end{align*}
   where $M$ is defined as in \eqref{def_M}.
   
   \begin{prop}\label{prop:Abf} 
   	Let $\{(a(\e),\omega(\e))\mid |\e|<\e_0\}\subset h^s(\R)\times I_{\lambda_*}$ be the local bifurcation curve found in Theorem~\ref{main2} corresponding to the bifurcation point $(0,\lambda_*)$. Then
   	\[
   	\dot \lambda(0)=0\qquad \mbox{and}\qquad 	\ddot \lambda(0)=-\frac{3\gamma}{4\int_0^\infty V_0(y) \phi_{k_*}^2\,dy}.
   	\]
   \end{prop}
   
   \begin{proof}
   	The proof follows essentially the lines of the proof of Proposition \ref{prop:bf}. We obtain that $\dot \lambda(0)=0$, which is due to the cubic character of the nonlinearity and
   	\[
   	\ddot \lambda(0)=-\frac{1}{3}\frac{\langle D_{aaa}^3G(0,\lambda_*)[e^{k_*},e^{k_*},e^{k_*}],e^{k_*}\rangle}{\langle D_{a\lambda}^2G(0,\lambda_*)e^{k_*},e^{k_*}\rangle}.
   	\]
   	Due to Lemma \ref{lem:prop_G} (i) and \eqref{transverse} the denominator is given by
   	\[
   	\langle D_{a\lambda}^2G(0,\lambda_*)e^{k_*},e^{k_*}\rangle = 2\int_0^\infty k^2 V_0(y) \psi_{k_*}^2\,dy,
   	\]
   	and the numerator reads
   	\[
   	\langle D_{aaa}^3G(0,\lambda_*)[e^{k_*},e^{k_*},e^{k_*}],e^{k_*}\rangle=-\frac{3}{2}\gamma k_*^2 (e^{k_*}*e^{k_*}*e^{k_*})_{k_*}.
   	\]
   	 Since $(e^{k_*}*e^{k_*}*e^{k_*})_{k_*}=-3$, as shown in Lemma~\ref{A2}, the statement follows.
   \end{proof}

   \bigskip
   
\section{Examples for distributional $\Gamma$} \label{S:distrib}

In what follows we prove Corollary~\ref{cor:p1_distrib} and Corollary \ref{cor:p2_distrib}, which state the existence of traveling waves for \eqref{eq:Q_trav} in the specific cases, when the potentials are given as in (P1) and (P2), respectively.
   
 \subsection{(P1) $V$ a $\delta$-potential on a positive background}
 \label{Ss:P1_distrib}
 We consider the particular case when $V_0=1$, $W=0$ so that we have a positive constant background potential with a multiple of a delta potential on top, i.e.,
 \[
	V(\lambda,y)=\lambda + \alpha\delta_0(y). 
\]
We verify the conditions $(\tilde L0)-(\tilde L3)$ of Theorem \ref{main2}; thereby proving part one of Corollary \ref{cor:p1_distrib}. 
Let us fix a wavenumber $k_*\in \N$ and a value $\lambda_*<1$.  We determine $\alpha>0$ from 
\begin{equation*}\label{eq:bifurcation_condition_2}
\alpha = \frac{2\sqrt{1-\lambda_*}}{k_*}.
\end{equation*}

Notice that the transversality condition and $(\tilde L0)$ are trivially satisfied. Moreover, the validity of $(\tilde L1)$ and $(\tilde L3)$ follow immediately from Remark~\ref{rem:pos}, since $1-\lambda V_0-W=1-\lambda>0$.  Condition $(\tilde L2)$ is exactly the same as $(L2)$ since our operator $L_\lambda^k$ is the same as the one considered in Corollary~\ref{cor:p1} of Case (P1). Since the choice of $\alpha$, $k_*$, $\lambda_*$ is the same as in Corollary~\ref{cor:p1} condition $(\tilde L1)$ holds and we are finished with treating this example.


Now, Corollary~\ref{cor:p1_distrib} follows from Theorem \ref{main2} and the  bifurcation formulas are an immediate consequence of Proposition~\ref{prop:Abf}.

\medskip

\subsection{(P2) $V$ a $\delta$-potential on a step background} 
\label{Ss:P2_distrib}
Now, we consider the case when $V_0=\textbf{1}_{|y|\geq b}$, $W=\beta \textbf{1}_{|y|<b}$ for some $b>0$ so that the potential $V$ is given by
\[
V(\lambda,y)=\lambda \textbf{1}_{|y|\geq b} + \beta \textbf{1}_{|y|<b} + \alpha \delta_0(y),
\]
Again we verify the conditions $(\tilde L0)-(\tilde L3)$ of Theorem \ref{main2}; thereby proving Corollary~\ref{cor:p2_distrib}. 
First we fix a wavenumber $k_*\in \N$ and a value $\lambda_*<1$.  According to Corollary~\ref{cor:p2_distrib} we have to distinguish between the case $\beta<1$, $\beta>1$, and $\beta=1$. Notice that the transversality condition and $(\tilde L0)$ are trivially satisfied for all $\beta \in \R$. 

\medskip

 Let us begin with the case $\beta<1$. The validity of $(\tilde L1)$ and $(\tilde L3)$ follow immediately from Remark~\ref{rem:pos}, since $1-\lambda V_0-W=(1-\lambda)\textbf{1}_{|y|\geq b}+(1-\beta\textbf{1}_{|y|<b})>0$. It remains to consider $(\tilde L2)$. But again the operator $L_\lambda^k$ is the same as the one considered in Corollary~\ref{cor:p1} of Case (P1) and the choice of $\alpha$ in Corollary~\ref{cor:p2_distrib} is exactly the same as in Corollary~\ref{cor:p2} of Case (P1). Hence $(\tilde L1)$ holds and this example is complete.

\medskip
 Next we consider the case $\beta>1$. Here we have made the choices 
\begin{equation*}\label{eq:con}
	\sqrt{\beta-1}b =\pi
\end{equation*}
and
\begin{equation*}\label{eq:bcon}
\alpha =\frac{2\sqrt{1-\lambda_*}}{k_*}.
\end{equation*}
We are left to verify $(\tilde L1)-(\tilde L3)$ of Theorem \ref{main2}. For $(\tilde L1)$ we need to consider the operator $L_{0,\lambda}^k=-\frac{d^2}{dy^2}+k^2(1-\lambda  \textbf{1}_{|y|\geq b} - \beta \textbf{1}_{|y|<b} ):H^2(\R)\to L^2(\R) $ which is self-adjoint with $\sigma_{ess}(L_{0,\lambda}^k)\subset [k^2(1-\lambda),\infty)$. Thus $0\in \rho (L_{0,\lambda}^k)$ if and only if $L_{0,\lambda}^k\phi=0$ for some $\phi \in H^2(\R)$ implies that $\phi=0$. In other words: we need to rule out that $L_{0,\lambda}^k$ has a zero eigenvalue. This can be seen from Lemma~\ref{eigenvalue_condition} in the Appendix if we set $\alpha=0$ (no delta potential in the equation) and $\mu=0$, i.e., $\tilde\lambda=\lambda$ and $\tilde\beta=\beta$. Moreover, we need to make the obvious changes $\sqrt{1-\beta}=\I\sqrt{\beta-1}$ and $\sinh(\I x)=\I\sin(x)$, $\cosh(\I x)=\cos(x)$. Following the ansatz \eqref{ansatz_ef} for the eigenfunction we obtain $c_0=d_0$ and $c_1=d_1$ due to the $C^1$-matching at $x=0$. Moreover, the choice of $\sqrt{\beta-1}b =\pi$ results in the invertible matrices 
\begin{equation*}
M_\pm = \begin{pmatrix} 
0  &  -e^{-k\sqrt{1-\tilde\lambda}b}\\
\sqrt{1-\tilde\beta}(-1)^k & \pm\sqrt{1-\tilde\lambda}e^{-k\sqrt{1-\tilde\lambda}b}
\end{pmatrix}.
\end{equation*}
Hence the conclusion $c_1=-d_1$ from Lemma~\ref{eigenvalue_condition} holds and leads to $c_1=d_1=0$. An inspection of the $C^1$-compatibility at $y=\pm b$ then yields $c_2=d_2=c_0=d_0=0$. Therefore, there is no zero-eigenvalue of $L_{0,\lambda}^k$ for any $k\in\N$ and any $\lambda\in (-\infty,1)$ and $(\tilde L1)$ holds.

\medskip

Concerning $(\tilde L2)$ we need to study a zero-eigenvalue of $L_\lambda^k$. The answer is again given by Lemma~\ref{eigenvalue_condition} in the Appendix since we already know the invertibility of the matrices $M_\pm$. Hence the eigenvalue condition is given by \eqref{this_makes_an_eigenvalue} with the obvious changes from the hyperbolic functions to the trigonometric function and reads
$$
\frac{k\alpha}{\sqrt{\beta-1}} = \frac{-\sqrt{\beta-1}\sin(k\sqrt{\beta-1}b)+\sqrt{1-\lambda}\cos(k\sqrt{\beta-1}b)}{\sqrt{\beta-1}\cos(k\sqrt{\beta-1}b)+\sqrt{1-\lambda}\sin(k\sqrt{\beta-1}b)}.
$$
In view of $\sqrt{\beta-1}b =\pi$ this reduces to 
$$
\alpha = \frac{2\sqrt{1-\lambda}}{k}
$$
and hence $L_\lambda^k$ has a zero-eigenvalue if and only if $k=k_*$ and $\lambda=\lambda_*$. Thus $(\tilde L2)$ holds. 
Finally, in order to verify $(\tilde L3)$, we compute the function $\phi_k$ which solve $L_\lambda^k\phi_k=0$ on $(0,\infty)$ with $\phi_k(0)=1$. From Lemma~\ref{eigenvalue_condition} we obtain 
\begin{equation*}
	\begin{cases}
		\phi_k(y,\lambda)=\cos(k\sqrt{\beta-1}y)+\I c_1\sin(k\sqrt{\beta-1}y),&\qquad y\in [0,b],\\
		\phi_k(y,\lambda)=c_2e^{-k\sqrt{1-\lambda}y},&\qquad y\geq b\\
	\end{cases}
\end{equation*}
with $c_1 = \frac{-\sqrt{1-\lambda}}{\I\sqrt{\beta-1}}$ and $c_2=e^{k\sqrt{1-\lambda}b}(-1)^k$. Computing the $L^2$-norm of $\phi_k$ we find that
 \begin{align*}
 \frac{1}{2}	\|\phi_k(\cdot,\lambda)\|_{L^2(\R)}^2 &= \int_0^b \left(\cos(k\sqrt{\beta-1}y)-\frac{\sqrt{1-\lambda}}{\sqrt{\beta-1}}\sin(k\sqrt{\beta-1}y)\right)^2\,dy + \int_b^\infty e^{2k\sqrt{1-\lambda}(b-y)}\,dy\\
 &=\frac{1}{2k\sqrt{1-\lambda}}+\frac{1}{2}\left(\frac{1-\lambda}{\beta-1}+1\right)b \leq C\left(1+\frac{1}{k}\right),
 	\end{align*}
 where the constant $C>0$ is independent of $k$  and can be chosen uniformly for $\lambda $ sufficiently close to $\lambda_*$. This shows the validity of $(\tilde L3)$. 
 
 \medskip
 
The last case to be considered is $\beta=1$. Also here, we are left to verify conditions $(\tilde L1)$, $(\tilde L2)$, and $(\tilde L3)$. First we find that in this case with $\tilde\lambda=\lambda$ and $\tilde \beta=\beta =1$ condition \eqref{this_makes_an_eigenvalue} is replaced by  
\begin{equation} \label{case_beta=1}
\frac{k\alpha}{2} = \frac{\sqrt{1-\lambda}}{1+\sqrt{1-\lambda}kb},
\end{equation}
which follows from a suitable adaptation of Lemma~\ref{eigenvalue_condition}. A zero eigenvalue of $L_{0,\lambda}^k$ correspond to values $k, \lambda$ satisfying \eqref{case_beta=1} with $\alpha=0$ which is impossible. Since $\sigma_{ess}(L_{0,\lambda}^k)=[k^2(1-\lambda),\infty)$ this shows that $(\tilde L1)$ holds. If we recall the definition of $\alpha$, i.e.,
$$
\alpha = \frac{2\sqrt{1-\lambda_*}}{k_*(1+\sqrt{1-\lambda_*}k_*b)}
$$
and compare with the $0$-eigenvalue condition \eqref{case_beta=1} we see that this ensures that $0$ is a (simple) eigenvalue of $L_\lambda^k$ if and only if $\lambda=\lambda_*$ and $k=k_*$. Hence, $(\tilde L2)$ holds. To see $(\tilde L3)$ we compute (also with the help of an adaptation of Lemma~\ref{eigenvalue_condition}) that the functions $\phi_k$ solving $L_\lambda^k\phi_k=0$ on $(0,\infty)$ with $\phi_k(0)=1$ are given by 
\begin{equation*}
	\begin{cases}
		\phi_k(y,\lambda)=1+c_1 y,&\qquad y\in [0,b],\\
		\phi_k(y,\lambda)=c_2e^{-k\sqrt{1-\lambda}y},&\qquad y\geq b\\
	\end{cases}
\end{equation*}
with $c_1 = \frac{-\sqrt{1-\lambda}k}{kb\sqrt{1-\lambda}+1}$ and $c_2=\frac{e^{k\sqrt{1-\lambda}b}}{kb\sqrt{1-\lambda}+1}$. From this we directly calculate that $\|\phi_k(\cdot,\lambda)\|_{L^2(\R)}^2=O(1)$ as $k\to \infty$ uniformly for $\lambda $ sufficiently close to $\lambda_*$. Hence, $(\tilde L3)$ holds.

\medskip

Now, Corollary \ref{cor:p2_distrib} follows from Theorem~\ref{main2}.

\bigskip

\appendix

\section{Auxiliary results}

\begin{lem}\label{A1} 
Let $A(x)=\sum_{k\in \N}a_k\sin(kx)$, then 
\[
A^3(x)=-\frac{1}{4}\sum_{k\in \N}(a*a*a)_k\sin(kx),
\]
where $a=(a_k)_{k\in\Z}$ is an infinite sequence with $a_k=-a_{-k}$ for all $k\in\Z$. The notation $(a*a*a)_k$ is used to denote the $k$-th entry in the sequence obtained by convolution $a*a*a$.
\end{lem}

\begin{proof}
If $a$ is a sequence as above then using $a_k=-a_{-k}$ for all $k$ we find that
\begin{align*}
\sum_{k\in \Z}\left(-\frac{1}{2}\mathrm{i} a_k\right) e^{\mathrm{i}kx}=\sum_{k\in \Z }\frac{1}{2}a_k\sin(kx)=\sum_{k\in \N}a_k\sin(kx)=A(x),
\end{align*}
and
\[
A^3(x)=\sum_{k\in \Z}\frac{1}{8}\mathrm{i}(a*a*a)_ke^{\mathrm{i}kx}.
\]
We are going to show that the Fourier coefficients $(a*a*a)_k$ are odd with respect to $k$. Notice first that
\begin{align*}
(a*a*a)_k=\sum_{j\in\Z}\left(\sum_{l\in \Z}a_{k-j-l}a_l\right)a_j. 
\end{align*}
We also have that
\begin{align*}
(a*a*a)_{-k}&= \sum_{j\in\Z}\left(\sum_{l\in \Z}a_{-k-j-l}a_l\right)a_j=-\sum_{j\in\Z}\left(\sum_{l\in \Z}a_{k+j+l}a_l\right)a_j\\
& = -\sum_{j\in\Z}\left(\sum_{l\in\Z}a_{k-j-l}a_{-l}\right)a_{-j}=-\sum_{j\in\Z}\left(\sum_{l\in\Z}a_{k-j-l}a_{l}\right)a_{j}=-(a*a*a)_k.
\end{align*}
From this we deduce that
\[
A^3(x)=\sum_{k\in \Z}\frac{1}{8}\mathrm{i}(a*a*a)_ke^{\mathrm{i}kx}=-\frac{1}{8}\sum_{k\in \Z}(a*a*a)_k\sin(xk)= -\frac{1}{4}\sum_{k\in \N}(a*a*a)_k\sin(xk).
\]
\end{proof}

\begin{lem}\label{A2} 
	Let $k\in \N$ and $e^{k_*}$ be a sequence such that $e^{k_*}_k=0$ if $k\neq \pm k_*$, $e^{k_*}_{k_*}=-e^{k_*}_{-k_*}=1$. Then,
	\[
	(e^{k_*}*e^{k_*}*e^{k_*})_k=\left\{ \begin{array}{lcl}1, &\mbox{if}& k=3k_*,\\ 
	-1, &\mbox{if}& k=-3k_*,\\-3,&\mbox{if}& k= k_*,\\
	3,&\mbox{if}& k= -k_*\end{array} \right.
	\]
\end{lem}

\begin{proof}
The convolution $e^{k_*}*e^{k_*}*e^{k_*}$ is given by
\begin{align*}
	(e^{k_*}*e^{k_*}*e^{k_*})_k&=\sum_{j\in\Z}\left(\sum_{l\in \Z}e^{k_*}_{k-j-l}e^{k_*}_l\right)e^{k_*}_j=\sum_{l\in \Z} e^{k_*}_{k-k_*-l}e^{k_*}_l-\sum_{l\in\Z}e^{k_*}_{k+k_*-l}e^{k_*}_l\\
	&= e^{k_*}_{k-2k_*} -2e^{k^*}_k+e^{k_*}_{k+2k_*} 
\end{align*}
and the claim follows.
\end{proof}

\begin{lem} \label{lem:regularity}
Let $a\in h^s(\R)$ for some $s\geq 0$ and define $\Phi(x,y) = \sum_{k\in\N} a_k \phi_k(y;\lambda) \sin(kx)$ for $x\in\T$ and $y\in \R$. Then 
\begin{itemize}
\item[(i)] $\Phi\in H^s(\T,L^2(\R)$
\item[(ii)] $\Phi \in H^{s-1}(\T,H^1(\R))$
\item[(iii)] $\Phi\in H^{s-2}(\T, H^2(0,\infty))$
\item[(iv)] $\Phi\in C(\R, H^{s-\frac{1}{2}}(\T))$
\item[(v)] $\Phi\in C^1(\R, H^{s-\frac{3}{2}}(\T))$
\end{itemize}
\end{lem}

\begin{proof}
	We verify that
\begin{align*}
\|\Phi\|_{H^s(\T,L^2(\R))}^2 & \leq C\sum_{k\in \N} a_k^2 k^{2s} \|\phi_k\|_{L^2(\R)}^2 \leq C \|a\|_{h^s(\R)}^2, \\
\|\Phi\|_{H^{s-1}(\T,H^1(\R))}^2 & \leq C\sum_{k\in \N} a_k^2 k^{2s-2} \|\phi_k'\|_{L^2(\R)}^2 \leq C \|a\|_{h^s(\R)}^2, \\
\|\Phi\|_{H^{s-2}(\T,H^2(0,\infty))}^2 & \leq C\sum_{k\in \N} a_k^2 k^{2s-4} \|\phi_k''\|_{L^2(0,\infty)}^2 \\
& \leq C(1+\|1-\lambda V_0-W\|_{L^\infty(\R)} \sum_{k\in \N} a_k^2 k^{2s} \|\phi_k\|_{L^2(0,\infty)}^2\leq C \|a\|_{h^s(\R)}^2, \\
\|\Phi\|_{C(\R,H^{s-\frac{1}{2}}(\T))}^2 & \leq C\sum_{k\in \N} a_k^2 k^{2s-1} \|\phi_k\|_{L^\infty(\R)}^2 \leq C\|a\|_{h^s(\R)}^2,\\
\|\Phi\|_{C^1(\R,H^{s-\frac{3}{2}}(\T))}^2 & \leq C\sum_{k\in \N} a_k^2 k^{2s-3} \|\phi_k'\|_{L^\infty(\R)}^2 \leq C\|a\|_{h^s(\R)}^2.
\end{align*}
\end{proof}

\begin{lem} \label{lem:algebra}
For $s\geq 1$ the space $h^s(\R)$ is a Banach algebra with respect to convolution.
\end{lem}

\begin{proof} In this proof we use the $l^1$-norm $\|a\|_{l^1(\R)}=\sum_{k\in\Z} |a_k|$ for a sequence $a=(a_k)_{k\in\Z}\in l^1(\R)$, i.e., the Banach space of all real sequences with finite $l^1$-norm. Due to convexity we have the inequality
$$
|k|^s \leq 2^{s-1}(|k-l|^s+|l|^s).
$$
Therefore, if $a, b \in h^s(\R)$ then 
$$
|k|^s (a*b)_k=|k|^s \left|\sum_{l\in \Z} a_{k-l}b_l\right|\leq 2^{s-1} \sum_{l\in\Z} |k-l|^s |a_{k-l}| |b_l| + |a_{k-l}| |l|^s |b_l|.
$$
Using the convolution inequality $\|\tilde a*\tilde b\|_{l^2} \leq \|\tilde a\|_{l^2} \|\tilde b\|_{l^1}$ once for $(\tilde a)_k=|k|^s |a_k|$, $(\tilde b)_k=|b_k|$ and once for $(\tilde a)_k=|a_k|$, $(\tilde b)_k = |k|^s|b_k|$ we get 
$$
\|a*b\|_{h^s(\R)} \leq 2^{s-1}(\|a\|_{h^s(\R)} \|b\|_{l^1(\R)} + \|a\|_{l^1(\R)}\|b\|_{h^s(\R)}).
$$
Finally, $a\in h^s(\R)$ implies $a\in l^1(\R)$ due to 
$$
\sum_{k\in \Z} |a_k| = \sum_{k\in\Z} |a_k| (|k|+1)\frac{1}{|k|+1} \leq C \|a\|_{h^1(\R)} \leq C \|a\|_{h^s(\R)}.
$$
\end{proof}

\begin{lem}\label{lem:spectrum} Let $L=-\frac{d^2}{dy} +q(y)$ with $q\in L^\infty(\R)$ be a self-adjoint operator on $L^2(\R)$ with domain $D(L)=H^2(\R)$. Then, for any $\alpha\in \R$, we have that $L_\alpha:=L+\alpha \delta_0$ is self-adjoint with domain $D(L_\alpha)=\{u\in H^1(\R)\cap \left(H^2(0,\infty)\cup H^2(-\infty,0)\right)\mid u^\prime(0_+)-u^\prime(0_-)=-\alpha u(0)\}$. Moreover for any $\alpha \in \R$ the following holds: 
\begin{itemize}
\item[(i)] For sufficiently large $\mu>0$ we have that $(L_\alpha+\mu)^{-1}: H^{-1}(\R) \to H^1(\R)$ is bounded.
\item[(ii)] $\sigma_{ess}(L_\alpha)=\sigma_{ess}(L)$.
\end{itemize}
\end{lem}

\begin{proof} A proof for the self-adjointness of $L_\alpha$ for any $\alpha \in \R$ is given in \cite{christ_stolz}. For (i) we first note that $L_\alpha$ is a semi-bounded self-adjoint operator so that $L_\alpha+\mu$ is a positive operator for $\mu>0$ sufficiently large. Its bilinear form $b_{L_\alpha+\mu}: H^1(\R)\times H^1(\R) \to \R$ is coercive and equivalent to the standard $H^1(\R)$-inner product. Therefore, any $f\in H^{-1}(\R)$ can be represented by a unique $u\in H^1(\R)$ such that $b_{L_\alpha+\mu}(u,\phi)=f(\phi)$ for any $\phi\in H^1(\R)$ and $b_{L_\alpha+\mu}(u,u)=\|\phi\|_{H^{-1}(\R)}^2$. This proves (i).

 For (ii) we may take $\lambda<0$ sufficiently negative such that $\lambda\in \rho(L)\cap \rho(L_\alpha)\cap \R$ since both $L, L_\alpha$ are semi-bounded from below. Using (i) we may also assume $\lambda$ sufficiently negative that $(\lambda-L_\alpha)^{-1}: H^{-1}(\R) \to H^1(\R)$ is bounded. By Weyl's criterion it is sufficient to show that the operator $W_\lambda:= (\lambda-L)^{-1}-(\lambda-L_\alpha)^{-1}: L^2(\R)\to L^2(\R)$ is compact in order to prove the statement. Since 
$$
W_\lambda = (\lambda-L_\alpha)^{-1}\circ((\lambda-L_\alpha)(\lambda-L)^{-1}-\Id)= \underbrace{(\lambda-L_\alpha)^{-1}}_{H^{-1}(\R)\to H^1(\R)\subset L^2(\R)} \circ \underbrace{(L-L_\alpha)}_{H^1(\R)\to H^{-1}(\R)} \circ\underbrace{(\lambda-L)^{-1}}_{L^2(\R)\to H^1(\R)}
$$
and since $L-L_\alpha=-\alpha\delta_0: H^1(\R)\to H^{-1}(\R)$ is a bounded operator with 1-dimensional range spanned by $\delta_0$ we see that $W_\lambda$ is indeed compact. This finishes the proof. 
\end{proof}

\begin{lem}\label{eigenvalue_condition} Let $\tilde\lambda, \tilde\beta <1$. Then the eigenvalue problem \eqref{ev_problem} is solvable for $\phi\in D(L_\lambda^k)$ if and only if \eqref{this_makes_an_eigenvalue} holds. In this case the eigenspace is one-dimensional.
\end{lem}

\begin{proof} Solutions of the differential equation in \eqref{ev_problem} have to be of the form
\begin{equation} \label{ansatz_ef}
\phi(y,\lambda)= 
	\begin{cases}
        c_2 e^{-k\sqrt{1-\tilde\lambda}y},&\qquad y\geq b, \\
		c_0\cosh(k\sqrt{1-\tilde\beta}y)+c_1\sinh(k\sqrt{1-\tilde\beta}y),&\qquad y\in [0,b],\\
		d_0\cosh(k\sqrt{1-\tilde\beta}y)+d_1\sinh(k\sqrt{1-\tilde\beta}y),&\qquad y\in [-b,0], \\
        d_2e^{k\sqrt{1-\tilde\lambda}y},&\qquad y\leq -b
	\end{cases}
\end{equation}
with $C^1$-compatibility conditions at $x=\pm b$ and continuity at $x=0$. The latter implies $c_0=d_0$ and the condition $\phi'(0+)-\phi'(0-)+k^2\alpha \phi(0)=0$ at $x=0$ translates into
\begin{equation} \label{bc_at_0}
k\sqrt{1-\tilde\beta}(c_1-d_1)+k^2\alpha c_0=0.
\end{equation}
The $C^1$-compatibility leads to the following set of four equations
\begin{align*}
	\begin{cases}
c_0\cosh(k\sqrt{1-\tilde\beta}b)+c_1\sinh(k\sqrt{1-\tilde\beta}b) & = c_2e^{-k\sqrt{1-\tilde\lambda}b}, \\
\sqrt{1-\tilde\beta} \bigl(c_0\sinh(k\sqrt{1-\tilde\beta}b)+c_1\cosh(k\sqrt{1-\tilde\beta}b)\bigr) & = -\sqrt{1-\tilde\lambda}c_2e^{-k\sqrt{1-\tilde\lambda}b},\\
d_0\cosh(k\sqrt{1-\tilde\beta}b)-d_1\sinh(k\sqrt{1-\tilde\beta}b) & = d_2e^{-k\sqrt{1-\tilde\lambda}b}, \\
\sqrt{1-\tilde\beta} \bigl(-d_0\sinh(k\sqrt{1-\tilde\beta}b)+d_1\cosh(k\sqrt{1-\tilde\beta}b)\bigr) & = \sqrt{1-\tilde\lambda}d_2e^{-k\sqrt{1-\tilde\lambda}b}.
\end{cases}
\end{align*}
These four equations can be written as 
$$
\begin{pmatrix}
M_+ &  0 \\
0 & M_-
\end{pmatrix}
\begin{pmatrix}
c_1 \\ c_2 \\ d_1 \\ d_2 
\end{pmatrix}
= \begin{pmatrix} 
-c_0 \cosh(k\sqrt{1-\tilde\beta}b) \\
-\sqrt{1-\tilde\beta} c_0\sinh(k\sqrt{1-\tilde\beta}b)\\
-c_0\cosh(k\sqrt{1-\tilde\beta}b)\\
\sqrt{1-\tilde\beta} c_0\sinh(k\sqrt{1-\tilde\beta}b)
\end{pmatrix} 
$$
with 
\begin{equation*} \label{def_matrix_M}
M_\pm = \begin{pmatrix} 
\pm\sinh(k\sqrt{1-\tilde\beta}b) &  -e^{-k\sqrt{1-\tilde\lambda}b}\\
\sqrt{1-\tilde\beta}\cosh(k\sqrt{1-\tilde\beta}b) & \pm\sqrt{1-\tilde\lambda}e^{-k\sqrt{1-\tilde\lambda}b}
\end{pmatrix}.
\end{equation*}
Since both $M_+$ and $M_-$ are invertible we see that w.l.o.g. we can choose $c_0=1$. Moreover, the structure of the linear systems yields that $c_1=-d_1$ and $c_2=d_2$. Finally, solving for $c_1, c_2$ we get 
\begin{align*}
c_1 &= -\frac{\sqrt{1-\tilde\beta}\sinh(k\sqrt{1-\tilde\beta}b)+\sqrt{1-\tilde\lambda}\cosh(k\sqrt{1-\tilde\beta}b)}{\sqrt{1-\tilde\beta}\cosh(k\sqrt{1-\tilde\beta}b)+\sqrt{1-\tilde\lambda}\sinh(k\sqrt{1-\tilde\beta}b)}, \\
c_2 &= e^{k\sqrt{1-\tilde\lambda}b}(c_1 \sinh(k\sqrt{1-\tilde\beta}b)+\cosh(k\sqrt{1-\tilde\beta}b)).
\end{align*}
Inserting $c_1$, $d_1=-c_1$ and $c_0=1$ into \eqref{bc_at_0} yields the condition \eqref{this_makes_an_eigenvalue} as claimed. It also shows that the eigenspace is one-dimensional (the only degree of freedom is the choice of $c_0$ which we took to be $1$ w.l.o.g.).
\end{proof}

\bigskip

\subsection*{Acknowledgments}
Funded by the Deutsche Forschungsgemeinschaft (DFG, German Research Foundation) -- Project-ID 258734477 -- SFB 1173.

\bigskip

\bibliographystyle{plain}

\bibliography{BIR_Traveling_waves_for_a_quasilinear_wave_equation_final.bib}

\begin{thebibliography}{10}

\bibitem{agrawal}
Govind~P. A.
\newblock {\em Nonlinear Fiber Optics (Sixth Edition)}.
\newblock Academic Press, 2019.

\bibitem{BDPR_2016}
T.~Bartsch, T.~Dohnal, M.~Plum, and W.~Reichel.
\newblock Ground states of a nonlinear curl-curl problem in cylindrically
  symmetric media.
\newblock {\em NoDEA Nonlinear Differential Equations Appl.}, 23(5):Art. 52,
  34, 2016.

\bibitem{Bartsch_Mederski_survey}
T.~Bartsch and J.~Mederski.
\newblock Nonlinear time-harmonic {M}axwell equations in domains.
\newblock {\em J. Fixed Point Theory Appl.}, 19(1):959--986, 2017.

\bibitem{christ_stolz}
C.~Sh. Christ and G.~Stolz.
\newblock Spectral theory of one-dimensional {S}chr\"{o}dinger operators with
  point interactions.
\newblock {\em J. Math. Anal. Appl.}, 184(3):491--516, 1994.

\bibitem{CrRab_bifurcation}
M.~G. Crandall and P.~H. Rabinowitz.
\newblock Bifurcation from simple eigenvalues.
\newblock {\em J. Functional Analysis}, 8:321--340, 1971.

\bibitem{cuenin_tretter}
J.-C. Cuenin and C.~Tretter.
\newblock Non-symmetric perturbations of self-adjoint operators.
\newblock {\em J. Math. Anal. Appl.}, 441(1):235--258, 2016.

\bibitem{dohnal_doerfler}
T.~Dohnal and W.~D\"{o}rfler.
\newblock Coupled mode equation modeling for out-of-plane gap solitons in 2{D}
  photonic crystals.
\newblock {\em Multiscale Model. Simul.}, 11(1):162--191, 2013.

\bibitem{dohnal_romani_doubly}
T.~Dohnal and G.~Romani.
\newblock Eigenvalue bifurcation in doubly nonlinear problems with an
  application to surface plasmon polaritons.
\newblock {\em NoDEA Nonlinear Differential Equations Appl.}, 28(1):Paper No.
  9, 30, 2021.

\bibitem{dohnal_romani_2021}
T.~Dohnal and G.~Romani.
\newblock Justification of the asymptotic coupled mode approximation of
  out-of-plane gap solitons in {M}axwell equations.
\newblock {\em Nonlinearity}, 34(8):5261--5318, 2021.

\bibitem{Friedman}
A.~Friedman.
\newblock {\em Partial differential equations of parabolic type}.
\newblock Prentice-Hall, Inc., Englewood Cliffs, N.J., 1964.

\bibitem{Kielhoefer}
H.~Kielh{\"{o}}fer.
\newblock {\em {Bifurcation theory}}, volume 156 of {\em Applied Mathematical
  Sciences}.
\newblock Springer, New York, second edition, 2012.

\bibitem{kohler_reichel}
S.~Kohler and W.~Reichel.
\newblock Breather solutions for a quasi-linear -dimensional wave equation.
\newblock {\em Studies in Applied Mathematics}, 148(2):689--714, 2022.

\bibitem{komornik}
V.~Komornik.
\newblock Uniformly bounded {R}iesz bases and equiconvergence theorems.
\newblock {\em Bol. Soc. Parana. Mat. (3)}, 25(1-2):139--146, 2007.

\bibitem{Mederski_2015}
J.~Mederski.
\newblock Ground states of time-harmonic semilinear {M}axwell equations in
  {$\mathbb{R}^3$} with vanishing permittivity.
\newblock {\em Arch. Ration. Mech. Anal.}, 218(2):825--861, 2015.

\bibitem{Mederski_survey_2}
J.~Mederski.
\newblock Nonlinear time-harmonic {M}axwell equations in {$\mathbb{R}^3$}:
  recent results and open questions.
\newblock In {\em Recent advances in nonlinear {PDE}s theory}, volume~13 of
  {\em Lect. Notes Semin. Interdiscip. Mat.}, pages 47--57. Semin. Interdiscip.
  Mat. (S.I.M.), Potenza, 2016.

\bibitem{Mederski_Reichel}
J.~Mederski and W.~Reichel.
\newblock Travelling waves for {M}axwell's equations in nonlinear and
  nonsymmetric media, preprint.
\newblock 2021.

\bibitem{PelSimWeinstein}
D.~E. Pelinovsky, G.~Simpson, and M.~I. Weinstein.
\newblock Polychromatic solitary waves in a periodic and nonlinear {M}axwell
  system.
\newblock {\em SIAM J. Appl. Dyn. Syst.}, 11(1):478--506, 2012.

\bibitem{reed_simon}
M.~Reed and B.~Simon.
\newblock {\em Methods of modern mathematical physics. {I}. {F}unctional
  analysis}.
\newblock Academic Press, New York-London, 1972.

\bibitem{stuart_1990}
C.~A. Stuart.
\newblock Self-trapping of an electromagnetic field and bifurcation from the
  essential spectrum.
\newblock {\em Arch. Rational Mech. Anal.}, 113(1):65--96, 1990.

\bibitem{stuart_1993}
C.~A. Stuart.
\newblock Guidance properties of nonlinear planar waveguides.
\newblock {\em Arch. Rational Mech. Anal.}, 125(2):145--200, 1993.

\bibitem{stuart_zhou_2010}
C.~A. Stuart and H.-S. Zhou.
\newblock Existence of guided cylindrical {TM}-modes in an inhomogeneous
  self-focusing dielectric.
\newblock {\em Math. Models Methods Appl. Sci.}, 20(9):1681--1719, 2010.

\end{thebibliography}

\end{document}